%% file: Project.tex
\documentclass[a4paper, final]{oxcsproject}

\usepackage[style=ieee, sorting=none, backend=biber, doi=false, isbn=false]{biblatex}
\newcommand*{\bibtitle}{References}

\title{Logic of Sets with Atoms}

\candidateno{Jake Masters}

\wordcount{9996}
\Degree{Part C - MMathCompSci Mathematics and Computer Science}
\degreedate{Trinity 2025}

\input{glossary}

\usepackage{svg}

\usepackage{siunitx}

\usepackage{enumitem}

\makeatletter
\newcommand\varitem[1]{\item[\textbf{A\arabic{enumi}\rlap{$#1$}.}]%
  \edef\@currentlabel{A\arabic{enumi}{$#1$}}}
\makeatother

\newcommand{\dotin}{\mathrel{\dot\in}}

\begin{document}

\begin{romanpages}



	\maketitle


	\begin{acknowledgements}
		\input{text/acknowledgements}
	\end{acknowledgements}

	\begin{abstract}
		\input{text/abstract}
	\end{abstract}

	\flushbottom

	\listoftodos

	\tableofcontents


	\printnoidxglossary[type=\acronymtype]

\end{romanpages}

\flushbottom

\include{text/sec1-introduction}
\include{text/sec2-orbit-finite-sets}
\include{text/sec3-permutation-models}
\include{text/sec4-computation}
\include{text/sec5-cardinal}
\include{text/sec6-future}

\startappendices
\include{text/appendix-examples-conditions}

\setlength{\baselineskip}{0pt} 

{\renewcommand*\MakeUppercase[1]{#1}
	\printbibliography[heading=bibintoc,title={\bibtitle}]}

\end{document}

%% file: glossary.tex
\usepackage[acronym]{glossaries}

\makenoidxglossaries

\newacronym{ac}{AC}{Axiom of Choice}
\newacronym{ac pure}{${AC}^{pure}$}{Axiom of Choice for the pure sets}
\newacronym{agc}{AGC}{Axiom of Global Choice}
\newacronym{agc pure}{${AGC}^{pure}$}{Axiom of Global Choice for the pure sets}
\newacronym{svc}{SVC}{Small Violations of Choice}
\newacronym{inner and outer models}{Inner\\Outer Models}{If $\mathcal{N},\mathcal{M}$ are structures in the language of set theory and $\mathcal{N}$ is a substructure of $\mathcal{M}$, then $\mathcal{N}$ is an inner model of $\mathcal{M}$ and $\mathcal{M}$ is an outer model of $\mathcal{N}$}

%% file: text/acknowledgements.tex
I am grateful to Bartek Klin for introducing me to computation with atoms and suggesting that the theory could be entirely embedded in the underlying set theory of permutation models. I would like to thank the members of Set Theory in the UK for their feedback and discussions about set theory, in particular: Richard Matthews for looking over early drafts of my proofs; and the questions about generalising some results to arbitrary supports, which inspired some discussion in Section \ref{sec:ext_prop}.

%% file: text/abstract.tex
Orbit-finite models of computation generalise the standard models of computation, to allow computation over infinite objects that are finite up to symmetries on atoms, denoted by $\mathbb{A}$. Set theory with atoms is used to reason about these objects. Recent work assumes that $\mathbb{A}$ is countable and that the symmetries are the automorphisms of a structure on $\mathbb{A}$. We study this set theory to understand generalisations of this approach. We show that: this construction is well-defined and sufficiently expressive; and that automorphism groups are adequate.
\newline Certain uncountable structures appear similar to countable structures, suggesting that the theory of orbit-finite constructions may apply to these uncountable structures. We prove results guaranteeing that the theory of symmetries of two structures are equal. Let: $PM(\mathcal{A})$ be the universe of symmetries induced by adding atoms in bijection with $\mathcal{A}$ and considering the symmetric universe; $\underline{\mathcal{A}}$ be the image of $\mathcal{A}$ on the atoms; and $\phi ^{PM(\mathcal{A})}$ be the relativisation of $\phi$ to $PM(\mathcal{A})$. \newline We prove that all symmetric universes of equality atoms have theory $Th(PM(\left\langle \mathbb{N}\right\rangle))$.
\newline We prove that for structures $\mathcal{A}$, `nicely' covered by a set of cardinality $\kappa$, there is a structure $\mathcal{B}\equiv\mathcal{A}$ of size $\kappa$ such that for all formulae $\phi(x)$ in one variable,
\begin{equation*}
    ZFC\vdash \phi(\underline{\mathcal{A}})^{PM(\mathcal{A})}\leftrightarrow\phi(\underline{\mathcal{B}})^{PM(\mathcal{B})}
\end{equation*}
As numerous statements about orbit-finite constructions are expressible in permutation models, this allows theorems about computation on countable atoms to hold for particular uncountable atoms.
In particular with $\mathcal{B}:=\left\langle\mathbb{Q},\leq\right\rangle$ and $\mathcal{A}:=\left\langle\mathbb{R},\leq\right\rangle$.

%% file: text/sec1-introduction.tex
\section{Introduction}
\label{sec:introduction}
\subsection{Background}
Set theory with atoms is an alternative to standard set theory, including elements that are not sets. Sets with atoms was introduced by Fraenkel \cite{fraenkel_begriff_1922} when attempting to prove that the negation of the axiom of choice is consistent.
\newline Fraenkel modified the standard $ZF$ axioms into the $ZFA$ axioms for set theory with atoms in a straightforward way. For example, extensionality axioms in both theories are:
\begin{itemize}
    \item $\mathbf{ZF}$: \textit{$x$ and $y$ have the same elements $\Rightarrow x=y$}
    \item $\mathbf{ZFA}$: \textit{$x$ and $y$ have the same elements and are both sets $\Rightarrow x=y$}
\end{itemize}
The $ZFA$ axioms are such that when restricted to the `pure universe', the class of all sets hereditarily containing no atoms, they become the standard $ZF$ axioms. We shall denote the set of atoms by $\mathbb{A}$ and assume that it is infinite.
\newline\newline Using the atoms, which he called urelements, Fraenkel restricted the universe to a class (the first Fraenkel model), which is still a model of $ZFA$ but where the axiom of choice fails.
\newline\newline More specifically, consider $Sym(\mathbb{A})$, the group of all bijections from $\mathbb{A}$ to $\mathbb{A}$, and define a canonical group action on the universe recursively. Furthermore, for a finite tuple of atoms $\mathbf{a}\in\mathbb{A}^{<\omega}$, say that $x$, an element of the universe, is supported by $\mathbf{a}$ if $\forall\pi\in Sym(\mathbb{A})$, if $\pi$ fixes $\mathbf{a}$ then it fixes $x$.
\newline We say $x$ is finitely supported if there is some finite tuple of atoms supporting it.
\newline\newline The first Fraenkel model is the class of all hereditarily finitely-supported elements. It is immediate that the axiom of choice fails in this model: no bijection between a pure set and the set of atoms is finitely supported (swapping any two atoms will fail to preserve the function), so the set of atoms will not be in bijection with an ordinal.
\newline\newline Fraenkel proved that if $ZFA$ is consistent then $ZFA+\neg AC$ is consistent, but failed to say anything about the consistency of $ZF+\neg AC$.
\newline\newline Mostowski, Lindenbaum, and Specker \cite{mostowski_uber_1939}\cite{specker_zur_1957} generalised Fraenkel's construction to define, for $\mathcal{M}$ a model of $ZFA$, $G\in\mathcal{M}$ a group acting on the atoms and $\mathcal{F}\in\mathcal{M}$ a normal filter of subgroups, the permutation model $PM(\mathcal{M},G,\mathcal{F})$.
\newline The construction is much that same as Fraenkel's construction, with $G$ playing the role of $Sym(\mathbb{A})$ and the filter providing a generalisation for being finitely supported. Specifically, we say that $x$ is supported by $\mathcal{F}$ if there is $H\in\mathcal{F}$ fixing $x$. The finite-support filter is the upward closure of $\left\{G_\mathbf{a}:\mathbf{a}\in\mathbb{A}^{<\omega}\right\}$ ($G_\mathbf{a}$ being the stabiliser of $\mathbf{a}$), which would be the filter for Fraenkel's original construction.
\newline\newline Sets with atoms have recently been rediscovered by computer scientists; originally by Gabbay \cite{gabbay_new_2002} as a way to embed $\alpha$-equivalence into the structure of the universe for a rigorous framework of working with objects up to $\alpha$-equivalence. Gabbay uses the atoms as the variable names (justifying the topic name of `Nominal Sets' \cite{pitts_nominal_2013}) and the first Fraenkel model. These are now called the `equality atoms', as the group is the automorphism group of the atoms in the structure of equality.\footnote{There is a topic called `Nominal Logic' \cite{pitts_nominal_2003}, which is unrelated to this thesis.}
\newline As before, the theory generalises, now to permutation models constructed by automorphism groups and finite-support filters \cite{bojanczyk_slightly_nodate}.
\newline An example, for countable atoms, is to biject the atoms with the rational numbers, $\mathbb{Q}$, to give them the structure $\left\langle \mathbb{Q},\leq\right\rangle$ and use its automorphism group to construct a finite-support permutation model. We say that $\left\langle \mathbb{Q},\leq\right\rangle$ generates this permutation model and $\left\langle \mathbb{A}\right\rangle$ generates the first Fraenkel model.
\newline When we assume that the set of atoms is countable and the structure on the atoms has nice properties (e.g. oligomorphic), we can define when a set is `orbit-finite' (finite up to symmetry) and define models of computation acting on, and defined by, orbit-finite sets e.g. automata \cite{bojanczyk_automata_2014} and Turing machines \cite{bojanczyk_turing_2013}. These orbit-finite models of computation are usually generalisations of the standard finite models of computation, and many standard results generalise. Importantly, when the structure has an effective theory, the behaviour of an orbit-finite deterministic Turing machine is computable \cite{bojanczyk_slightly_nodate}.
\newline\newline When defining these notions of computation, we refer to the automorphism group of the structure; this means that we express properties of computation as first-order statements in the external universe, not in the permutation model. Concerningly, since the group, $G$, will not usually exist inside the permutation model, $PM(\mathcal{M},G,\mathcal{F})$, we may not be able to express important properties in the first-order logic of the permutation model e.g. being orbit-finite.
\newline However, Blass and Brunner \cite{blass_injectivity_1979}\cite{brunner_fraenkel-mostowski_1990} showed that many properties refering to the group can be stated in the first-order logic of the permutation model. In particular, Blass showed that finite-support models can express the property of being orbit finite, suggesting that most interesting properties of orbit-finite computation are expressible by the first-order logic of the finite-support permutation models.

\newpage
\subsection{Motivation}
We care about the logic of permutation models generated by specific structures. If we identify permutation models as elementarily equivalent then we can transfer theorems between them.
\newline The obvious example would be `sets are orbit-finite iff they are definable', which is the main component in orbit-finite computation. The theorem assumes that the structure on the atoms,  is countable and oligomorphic, but if we can show that some finite-support permutation model generated by an uncountable and oligomorphic structure on atoms is elementarily equivalent to a finite-support model generated by some countable structure, then the theorem still applies. In fact, this theorem justifies that this identification might be possible; its proof can be adapted to work for uncountable structures given that some countable structure `nicely covers' (Theorem \ref{thm:transfer_down_sub}) it. For example, $\left\langle \mathbb{Q},\leq\right\rangle$ satisfies the standard conditions for the theorem and `nicely covers' $\left\langle \mathbb{R},\leq\right\rangle$; there is no obvious way to distinguish the theories of the permutation models generated by $\left\langle \mathbb{Q},\leq\right\rangle$ and $\left\langle \mathbb{R},\leq\right\rangle$ so are they the same? If two permutation models are elementarily equivalent then the true statements about orbit-finite constructions on the models' generating structures, that are expressible in the permutation models, are the same.
\newline\newline These give us our main question: \textbf{When can we show that finite-support permutation models generated by different structures on the atoms are isomorphic or elementarily equivalent?}
\newline\newline Additionally, there are some questions regarding expressibility and well-defined-ness. Indeed, we have been referring to \textbf{the} permutation model induced by some structure given a pure universe. Any model of $ZF$ can be extended by atoms of the correct cardinality and a corresponding permutation model can be taken, so one such permutation model can always be constructed. However, we want to know if it's unique.
\newline\newline When we talk about orbit-finite constructions, we refer to the generating structure and its automorphism group. The group is usually excluded from the permutation model and if $\mathcal{N}$ is a finite-support permutation model of both $\mathcal{M}$ and $\mathcal{M}'$ the structures, $\mathcal{A}\in\mathcal{M}$ and $\mathcal{A}'\in\mathcal{M}'$, that generate $\mathcal{N}$ may differ (although both $\mathcal{A},\mathcal{A}'\in\mathcal{N}$). We want to know what we can express in the first-order logic of the permutation models.

\newpage
\subsection{Contributions}
In this thesis we study the logic of permutation models and justify its expressiveness regarding orbit-finite constructions.
We will summarise contributions at the beginning of each section, and briefly state our main results here.
\begin{itemize}
    \item We generalise the notion of orbit-finite computation:
    \begin{itemize}
        \item We consider uncountable structures on atoms.
        \item We show that every finite-support permutation model is induced by an automorphism group (Corollary \ref{thm:adequacy_automorphism}).
    \end{itemize}
    \item We characterise equality of permutation models on the same supports but differing pure structures (Theorem \ref{thm:different_structure_model}).
    \item We strengthen a result of Brunner \cite{brunner_fraenkel-mostowski_1990} that being a finite-support permutation model is expressible, by replacing the assumption of global choice by a weak choice axiom, $SVC$ (Corollary \ref{thm:finite_support_fo_improved}).
    \item We justify that numerous interesting statements and definitions about orbit-finite constructions are expressible in corresponding finite-support permutation models (Section \ref{sec:computation}).
    \item We prove a collection of theorems in Section \ref{sec:cardinal} that allow us to conclude when a finite-support permutation model is a finite-support permutation model in a universe where the atoms have a different cardinality.
    \begin{itemize}
        \item We provide a method of determining a structure inducing the permutation model in the new universe (Theorem \ref{thm:full_transfer}).
    \end{itemize}
    \item We identify all `first Fraenkel models' on the same pure universe (Theorem \ref{thm:first_fraenkel_unique}).
    \item We identify the theories of the finite-support permutation models induced by $\left\langle \mathbb{R},\leq\right\rangle$ and $\left\langle \mathbb{Q},\leq\right\rangle$ (allowing reference to the inducing structures) (Example \ref{examples:R}).
    \begin{itemize}
        \item We conclude that numerous theorems about orbit-finite structures on $\left\langle \mathbb{Q},\leq\right\rangle$ hold for $\left\langle \mathbb{R},\leq\right\rangle$ and vice-versa.
    \end{itemize}
\end{itemize}

%% file: text/sec2-orbit-finite-sets.tex
\section{Orbit-Finite Sets}
\label{sec:orbit-finite-sets}

\subsection{Summary}
In this section we:
\begin{itemize}
    \item Introduce set theory with atoms.
    \begin{itemize}
        \item Construct a model of $ZFA$ in a model of $ZF$ and vice-versa.
    \end{itemize}
    \item Introduce orbit-finite sets.
    \item Recall standard theorems about definability.
\end{itemize}

\subsection{Sets with Atoms}
\label{sec:sets_with_atoms}

Standard set theory formalises mathematics by allowing us to talk about a `universe' and its elements, the sets. In its view, sets can only contain sets, as they are all that exist.
\newline The typical choice is to use the $ZF$ axioms and the first order language $\left\{\in\right\}$, containing a single binary predicate symbol.
\newline\newline This can seem somewhat unnatural, as we might want to consider indivisible object, \textbf{atoms}. Sets with atoms, in the language $\mathcal{L}=\left\{\in,\mathbb{A}\right\}$ with a single binary predicate and a single unary predicate, resolve this issue.
\newline\newline As with standard set theory, we work using a set of axioms, \textbf{Zermelo-Fraenkel set theory with atoms (ZFA)}.\cite{jech_axiom_1973} 
\begin{definition}[ZFA] The following are the axioms of $ZFA$:  
\newline\indent\textbf{Extensionality} For all sets $x$ and $y$ i.e. $\neg\mathbb{A}(x)$ and $\neg\mathbb{A}(y)$, if for all $z$ ($z\in x$ iff $z\in y$) then $x=y$.
\newline\indent\textbf{Empty Set} There is a set containing no elements, denoted $\emptyset$.
\newline\indent\textbf{Pairing} For all elements $x$ and $y$, $\left\{x,y\right\}$ is a set.
\newline\indent\textbf{Union} For all sets $x$, $\bigcup x$ is a set.
\newline\indent\textbf{Power Set} For all sets $x$, there is a set $\mathcal{P}(x)$ containing precisely its subsets.
\newline\indent\textbf{Infinity} There is a set $x$, s.t. for $y\in x$, $y\cup \left\{y\right\}\in x$.
\newline\indent\textbf{Foundation} For all non-empty sets $x$, there is $y\in x$ with $y\cap x=\emptyset$
\newline\indent\textbf{Comprehension} For all $\mathcal{L}$-formulae $\phi(x_0,\dots,x_{n-1},y)$, for all elements $a_0,\dots,a_n$, $\left\{y\in a_n:\phi(a_0,\dots,a_{n-1},y)\right\}$ is a set.
\newline\indent\textbf{Replacement} For all $\mathcal{L}$-formulae $\phi(x_0,\dots,x_{n-1},x,y)$ and elements $a_0,\dots,a_n$, if for $x\in a_n$ there is unique $y$ s.t. $\phi (a_0,\dots,a_n,y)$ then $\left\{y:\exists x\in a_n.\phi(a_0,\dots,a_{n-1},x,y)\right\}$ is a set.
\newline\indent\textbf{Empty Atoms} For all atoms $a$, $a$ contains no elements.
\newline\indent\textbf{Set of Atoms} There is a set containing precisely the atoms, denoted $\mathbb{A}$.

\end{definition}

Sometimes we want to refer to collections of elements which may not be sets. These are called classes.

\begin{definition}[Class]
    Let $\phi(x_0,\dots,x_{n-1},y)$ be an $\mathcal{L}$-formula and $a_0,\dots,a_{n-1}$ be elements.
    \newline The collection $U=\left\{y:\phi(a_0,\dots a_{n-1},y)\right\}$ is a class. We use the standard abuse of notation of $x\in U$ to mean $\phi(a_0,\dots,a_{n-1},x)$.
\end{definition}

\begin{example}
    Denote by $\mathcal{V}(\mathbb{A})$ the class of all elements.
    \begin{equation*}
        \mathcal{V}(\mathbb{A}):=\left\{x:x=x\right\}
    \end{equation*}
    For an $\mathcal{L}$-formula $\phi(x_0,\dots,x_{n-1})$ and $a_0,\dots,a_{n-1}\in\mathcal{V}(\mathbb{A})$, $\phi(a_0,\dots,a_{n-1})$ holds in the universe iff $\left\langle\mathcal{V}(\mathbb{A}),\in,\mathbb{A}\right\rangle\models\phi(a_0,\dots,a_{n-1})$, where `$\models$' takes the standard meaning. We will sometimes just write $\mathcal{V}(\mathbb{A})\models\phi(a_0,\dots,a_{n-1})$ when the relations are clear.
\end{example}
\begin{example}
    \label{examples:ordinals}
    Denote by $\mathbf{On}$ the class of all ordinal numbers.
    \begin{equation*}
        \mathbf{On}:=\left\{\alpha:\alpha\text{ is well-ordered by}\in\text{and for all }y\in x\in \alpha, y\in \alpha\right\}
    \end{equation*}
\end{example}

We will want to talk about formulae holding in a class, so we can localise formulae.

\begin{definition}[Localisation]
    Let $U$ be a class and $\phi$ be an $\mathcal{L}$-formula.
    \newline The localisation of $\phi$ to $U$, $\phi^U$, is attained by replacing all occurances of $\forall x$ and $\exists x$ by $\forall x\in U$ ($\forall x\left(x\in U\rightarrow\dots\right)$) and $\exists x\in U$ ($\exists x\left(x\in U\wedge\dots\right)$).
\end{definition}

\begin{theorem}[\cite{jech_axiom_1973}]
    Let $U$ be a class, $\phi(x_0,\dots,x_{n-1})$ be a formula, and $a
    _0,\dots,a_{n-1}\in U$.
    \begin{equation}
        {\phi(a_0,\dots,a_{n-1})}^U\iff U\models\phi(a_0,\dots,a_{n-1})
    \end{equation}
\end{theorem}

\begin{definition}[Absoluteness \cite{jech_set_2013}]
\label{def:absoluteness}
    Let $\phi(x_0,\dots,a_{n-1})$ be a formula and $\mathcal{N}\subseteq\mathcal{M}$ be classes. We define $\phi$ to be absolute between $\mathcal{N}$ and $\mathcal{M}$ if, for $a_0,\dots,a_{n-1}\in\mathcal{N}$, $\left\langle\mathcal{N},\in\right\rangle\models\phi(a_0,\dots,a_{n-1})$ iff $\left\langle\mathcal{M},\in\right\rangle\models\phi(a_0,\dots,a_{n-1})$.
    \newline We can define the L\'{e}vy hierarchy ($\Sigma_n$, $\Pi_n$, and $\Delta_n$ for $n\in \mathbb{N}$) on logical formulae based upon their quantifier complexity. $\Delta_1$ formulae are absolute between transitive classes.
    \newline For a collection of formulae, $\Gamma$, and a collection of axioms, $S$, we define $\Gamma^S$ to be the formulae equivalent to a formula in $\Gamma$ under the assumptions $S$.
\end{definition}

Although the universe contains atoms, we can still consider the elements that contain no atoms, contain no sets containing atoms, and so on.

\begin{definition}[Transitive Closure]
    For $x\in\mathcal{V}(\mathbb{A})$, 
    \begin{equation*}
        TC(x):=x\cup\bigcup x\cup\bigcup\bigcup x\cup \dots
    \end{equation*}
    In other words, $TC(x)$ is the smallest set s.t. $x\subseteq TC(x)$ and $y\in TC(x)\rightarrow y\subseteq TC(x)$.
\end{definition}

\begin{definition}[Pure Sets]
    \label{def:pure_sets}
    We say that $x$ is pure iff $\forall y\in\left\{x\right\}\cup TC(x)\space \neg\mathbb{A}(y)$
    \newline Denote by $\mathcal{V}$ the class of all pure sets.
    \newline For a formula, $\phi$, we may write $\phi^{pure}$ for $\phi^\mathcal{V}$.
\end{definition}

If $\mathcal{V}(\mathbb{A})\models ZFA$, then $\mathcal{V}\models ZF$. So if $ZFA$ is consistent, then $ZF$ is consistent.

\begin{theorem}[Cumulative Hierarchy \cite{jech_axiom_1973}]
    For a set $x$, we define $\mathcal{V}(x)_\alpha$ for $\alpha$ an ordinal by recursion.
    \begin{align*}
        \mathcal{V}(x)_0&:=x\\
        \mathcal{V}(x)_{\alpha+1}&:=\mathcal{P}(\mathcal{V}(x)_\alpha)\cup \mathcal{V}(x)_\alpha\\
        \mathcal{V}(x)_\lambda&:=\bigcup_{\alpha<\lambda}\mathcal{V}(x)_\alpha&\text{for }\lambda\text{ a limit ordinal}\\
        \mathcal{V}(x)&:=\bigcup_{\alpha}\mathcal{V}(x)_\alpha
    \end{align*}
    Then $\mathcal{V}(\mathbb{A})$ is the universe and $\mathcal{V}(\emptyset)=\mathcal{V}$.
\end{theorem}

Not only can we construct a model of $ZF$ inside any model of $ZFA$, but we can also construct a model of $ZFA$ inside any model of $ZF$ with the pure universe being the model of $ZF$ and the atoms being in bijection with a given pure set. We sketch the construction below; see \cite{blass_freyds_1989} for the full detail.

\begin{definition}
    \label{def:class_adding_atoms}
    Let $\mathcal{V}\models ZF$, and $x\in\mathcal{V}$.
    Define the class $\mathcal{V}^*(x)$ by recursion:
    \begin{align*}
        \mathcal{V}^*(x)_0&:=x\times\left\{1\right\}\\
        \mathcal{V}^*(x)_{\alpha+1}&:=\mathcal{P}(\mathcal{V}^*(x)_\alpha)\times\left\{0\right\}\cup\mathcal{V}^*(x)_\alpha\\
        \mathcal{V}^*(x)_\lambda&:=\bigcup_{\alpha<\lambda}\mathcal{V}^*(x)_\alpha&\text{for }\lambda\text{ a limit ordinal}\\
        \mathcal{V}^*(x)&:=\bigcup_\alpha\mathcal{V}^*(x)_\alpha
    \end{align*}
    Define a binary predicate, $\dotin$, and a unary predicate, $\mathbb{A}'$, on $\mathcal{V}^*(x)$ by:
    \begin{align*}
        w \dotin y&\iff y=\left\langle z,0\right\rangle\wedge w\in z\\
        \mathbb{A}'(w)&\iff w=\left\langle z,1\right\rangle
    \end{align*}
    Define a class function $F:\mathcal{V}\rightarrow\mathcal{V}^*(x)$ by recursion:
    \begin{align*}
        F(y)&:=\left\langle\left\{F(z):z\in y\right\},0\right\rangle
    \end{align*}
\end{definition}
\begin{theorem}
\label{thm:zfa_in_zf}
    \begin{enumerate}
        \item \label{thm:zfa_in_zf_1}$\left\langle \mathcal{V}^*(x),\dotin,\mathbb{A}'\right\rangle\models ZFA$
        \item\label{thm:zfa_in_zf_2} $F$ is an isomorphism between $\mathcal{V}$ and the pure universe of $\mathcal{V}^*(x)$
        \item\label{thm:zfa_in_zf_3} $\left\langle \mathcal{V}^*(x),\dotin,\mathbb{A}'\right\rangle\models \left|\mathbb{A}\right|=\left|F(x)\right|$
        \item \label{thm:zfa_in_zf_4}$\mathcal{V}\models AC\iff\left\langle \mathcal{V}^*(x),\dotin,\mathbb{A}'\right\rangle\models AC$
    \end{enumerate}
\end{theorem}
See \cite{blass_freyds_1989} and \cite{jech_axiom_1973} for the full details. We sketch the proof here.
\begin{proof}
\begin{itemize}
    \item [\eqref{thm:zfa_in_zf_1}] Each axiom of ZFA in $\dotin, \mathbb{A}'$ can be translated into a statement, which is either (almost) an instance of an axiom of ZF or states that atoms are empty, which is immediate from $\dotin$. We give an example.
    \newline\textbf{Comprehension} For all $\mathcal{L}$-formulae $\phi(x_0,\dots,x_{n-1},y)$, for all elements $a_0,\dots,a_n$, $\left\{y\in a_n:\phi(a_0,\dots,a_{n-1},y)\right\}$ is a set.
    \newline Let $\phi$ be as above and $a_0,\dots,a_n\in \mathcal{V}^*(x)$. Let $\phi'$ be $\phi$ with $\dotin$ for $\in$ and $\mathbb{A}'$ for $\mathbb{A}$.
    \newline By comprehension,
    \begin{equation*}
        X=\left\{y\in \pi(a_n):\phi'(a_0,\dots,a_{n-1},y)\right\}\subseteq\mathcal{V}^*(x)
    \end{equation*}
    is a set (where $\pi$ is the projection onto the first coordinate).
    So $\left\langle X,0\right\rangle\in\mathcal{V}^*(x)$ is the desired element.
    \item [\eqref{thm:zfa_in_zf_2}] The pure universe of $\mathcal{V}^*(x)$ is $\mathcal{V}^*(\emptyset)$. Mostowski's collapsing lemma \cite{jech_axiom_1973} gives an isomorphism between $\mathcal{V}^*(\emptyset)$ and its image, and the construction gives the isomorphism as being inverse to $F$, so $F$ is an isomorphism.
    \item [\eqref{thm:zfa_in_zf_3}] Define $f:x\times\left\{1\right\}\rightarrow F(x)\times \left\{0\right\}$ by $f(\left\langle a,1\right\rangle)=\left\langle F(a),0\right\rangle$ and define
    \begin{equation*}
        f'=\left\{\left\{\left\{l\right\},\left\{l,r\right\}\right\}\times\left\{0\right\}:\left\langle l,r\right\rangle\in f\right\}\times\left\{0\right\}
    \end{equation*}
    We have that $f'\in\mathcal{V}^*(x)$ and $\mathcal{V}^*(x)\models \textit{`$f'$ is a bijection from $\mathbb{A}$ to $F(x)$'}$.
    \item [\ref{thm:zfa_in_zf_4}] Given a set in $\mathcal{V}^*(x)$, we can take a choice function in $\mathcal{V}$ and use a similar trick as above to move the choice function into $\mathcal{V}^*(x).$
\end{itemize}
\end{proof}
So $ZF$ and $ZFA$ are equiconsistent. Notice that given a model of $ZFA$, we can perform this construction within the pure universe.
\begin{theorem}
    \label{thm:atoms_iso_pure_class}
    Let $\mathcal{V}(\mathbb{A})\models ZFA$ and suppose $x\in\mathcal{V}(\mathbb{A})$ with $\left|x\right|=\left|\mathbb{A}\right|$, then
    \begin{equation*}
        \mathcal{V}(\mathbb{A})\cong \mathcal{V}^*(x)
    \end{equation*}
\end{theorem}
\begin{proof}
    Let $f:x\rightarrow\mathbb{A}$ be a bijection.
    Define $\pi:\mathcal{V}^*(x)\rightarrow\mathcal{V}(\mathbb{A})$ by:
    \begin{align*}
        \pi(\left\langle a,1\right\rangle)&:=f(a)\\
        \pi(\left\langle y,0\right\rangle)&:=\left\{\pi(z):z\in y\right\}
    \end{align*}
    By Mostowski's collapsing lemma for $ZFA$ \cite{blass_freyds_1989}, $\pi$ is an isomorphism onto its image.
    We show that $\pi$ is surjective by well-founded-ness.
    \newline For $a\in\mathbb{A}$, $a=\pi(\left\langle f^{-1}(a),1\right\rangle)$ so $\mathbb{A}\subseteq \pi[\mathcal{V}^*(x)]$.
    \newline Suppose that $z\subseteq\pi[\mathcal{V}^*(x)]$ so $\left\{\pi^{-1}\left(y\right):y\in z\right\}\in \mathcal{V}$ is a well-defined set. Let $Z=\left\langle\left\{\pi^{-1}(y):y\in z\right\},0\right\rangle$. Note that $Z\in \mathcal{V}^*(x)$ and $\pi(Z)=z$. We now deduce that $\mathcal{V}(\mathbb{A})\subseteq\pi[\mathcal{V}^*(x)]$ by well-founded-ness so $\pi$ is surjective.
\end{proof}
It might have been possible that there are distinct models of $ZFA$ with the same pure universe with the atoms having the same cardinality. Theorem \ref{thm:atoms_iso_pure_class} shows that, fixing a pure universe and cardinality on the atoms, all models of $ZFA$ are isomorphic to the same class in the pure universe. So extending a model of $ZF$ to a model of $ZFA$ by adding atoms in bijection with a pure set is well-defined and unambiguous. Also, $\pi$ allows us to transform any sentence on the entire universe into a sentence on the pure universe. For every $\mathcal{L}-$sentence $\phi$, there is a $\left\{\in\right\}$-formula $\phi'(x)$ (attained by replacing $\in$ by $\dotin$, $\mathbb{A}$ by $\mathbb{A}'$, and bounding all quantifiers by $\mathcal{V}^*(x)$) s.t.
\begin{equation*}
    ZFA\models\forall x\in\mathcal{V}\left(\left|\mathbb{A}\right|=\left|x\right|\rightarrow\left(\phi \leftrightarrow{\phi'}^\mathcal{V}(x)\right)\right)
\end{equation*}

\begin{example}
    All atoms contain no elements.
    \begin{equation*}
        \forall x\forall y \left(\mathbb{A}\left(x\right)\rightarrow\neg y\in x\right)
    \end{equation*}
    All atoms of $\mathcal{V}^*(x)$ contain no elements.
    \begin{equation*}
        \forall x\in \mathcal{V}^*(x)\forall y\in\mathcal{V}^*(x) \left(\mathbb{A}'\left(x\right)\rightarrow\neg y\dotin x\right)
    \end{equation*}
\end{example}

\begin{example}
    The atoms are in bijection with the naturals.
    \begin{equation*}
        \exists f\in\mathcal{P}(\mathbb{A}\times\mathbb{N}).\left(\forall n\in\mathbb{N}\exists!a\in\mathbb{A}.\left\langle a,n\right\rangle\in f\wedge\forall a\in\mathbb{A}\exists!n\in\mathbb{N}.\left\langle a,n\right\rangle\in f\right)
    \end{equation*}
    The atoms of $\mathcal{V}^*(x)$ are in bijection with the naturals of $\mathcal{V}^*(x)$.
    \begin{equation*}
        \exists f\dotin\mathcal{P}^{\mathcal{V}^*(x)}(\mathbb{A}'\times^{\mathcal{V}^*(x)}\mathbb{N}^{\mathcal{V}^*(x)}).\left(
        \begin{array}{l}
             \forall n\dotin\mathbb{N}^{\mathcal{V}^*(x)}\exists!a\dotin\mathbb{A}'.\left\langle a,n\right\rangle^{\mathcal{V}^*(x)}\dotin f \wedge\\
             \forall a\dotin\mathbb{A}'\exists!n\dotin\mathbb{N}^{\mathcal{V}^*(x)}.\left\langle a,n\right\rangle^{\mathcal{V}^*(x)}\dotin f
        \end{array}
        \right)
    \end{equation*}
    The superscript $X^{\mathcal{V}^*(x)}$ is used where the set $X$ is uniquely defined by a formula and denotes that $X^{\mathcal{V}^*(x)}$ is uniquely defined by the formula translated to $\mathcal{V}^*(x)$.
\end{example}

\subsection{Orbit-Finiteness}

Unlike in $ZF$, where the only class automorphism on the universe is the identity, we can extend permutations of $\mathbb{A}$ to automorphisms of the universe.

\begin{definition}[\cite{bojanczyk_slightly_nodate}]
    For $\pi\in Sym(\mathbb{A})$, define:
    \begin{align*}
        \overline\pi&:\mathcal{V}(\mathbb{A})\rightarrow\mathcal{V}(\mathbb{A})\\
        \overline\pi(a)&=\pi(a)&\text{for }a\in\mathbb{A}\\
        \overline\pi(x)&=\left\{\overline\pi(y):y\in x\right\}&\text{for }x\notin\mathbb{A}
    \end{align*}
    Abusing notation, we will write $\pi$ for $\overline\pi$.
\end{definition}
Let $G\leqslant Sym(\mathbb{A})$.
\begin{definition}[Support]
\label{def:support}
    Let $\mathbf{a}\in\mathbb{A}^{<\omega}$.
    \newline We say that $\mathbf{a}$ supports $x$ if for all $\pi\in G$ with $\pi(\mathbf{a})=\mathbf{a}$, $\pi(x)=x$.
    \newline If there exists $\mathbf{a}\in\mathbb{A}^{<\omega}$ supporting $x$, then we say that $x$ is finitely supported.
\end{definition}
\begin{remark}
    The definition of `$\mathbf{a}$ supports $x$' is implicitly parametrised by a group. The group will either be clear from the context or explicitly stated.
\end{remark}
\begin{remark}
    If $\mathbf{a}\subseteq\mathbf{b}\in\mathbb{A}^{<\omega}$ and $\mathbf{a}$ supports $x$ then so does $\mathbf{b}$.
\end{remark}

\begin{example}
\label{examples:support}
    Let $G=\left\{id_{\mathbb{A}}\right\}$. Then:
    \begin{itemize}
        \item $\forall x\in\mathcal{V}(\mathbb{A})$, $x$ has empty support (called equivariant).
    \end{itemize}
    Let $\left\{\underline0,\underline1,\dots\right\}= \mathbb{A}$ and $G=Sym(\mathbb{A})=Aut(\left\langle\mathbb{A}\right\rangle)$. Then:
    \begin{itemize}
        \item $\underline0$ is supported by $\underline0$.
        \item $\underline0$ is supported by $\left(\underline0,\underline1\right)$.
        \item $\left\{\emptyset,\left\{\emptyset\right\}\right\}$ is equivariant.
        \item $\left\{\left\{a,a'\right\}:a,a'\in\mathbb{A}\right\}$ is equivariant.
        \item For $\sigma\in Sym\left(\mathbb{N}\right)$, $x_\sigma:=\left\{\underline{\sigma 0},\left\{\underline{\sigma 1}\right\},\left\{\left\{\underline{\sigma2}\right\}\right\},\dots\right\}$ is not finitely supported.
        \item $\left\{x_\sigma:\sigma\in Sym\left(\mathbb{N}\right), \sigma5=5\right\}$ is supported by $\underline5$ but is not hereditarily finitely supported.
    \end{itemize}
    Let $\left\{\underline q:q\in\mathbb{Q}\right\}=\mathbb{A}$ and $G=Aut\left(\left\langle\mathbb{A},\leq\right\rangle\right)$:
    \begin{itemize}
        \item $\left(\underline0,\underline{\frac{5}{2}}\right]\cup\left[\underline3,\infty\right)$ has support $\left(\underline0,\underline{\frac{5}{2}},\underline3\right)$.
        \item $\bigcup_{n\in\mathbb{Z}}\left(\underline n, \underline {\left(n+0.5\right)}\right)$, $\left\{\underline q:q\in\mathbb{Q}\wedge q\leq \sqrt{2}\right\}$ are not finitely supported.
    \end{itemize}
\end{example}
\begin{definition}
    \label{def:support_structure}
    As we will be working with the automorphism groups of structures, for a structure $\mathcal{A}$ on the atoms, we use $\mathcal{A}$ in place of $Aut(\mathcal{A})$ in Definition \ref{def:support}.
\end{definition}
As in some of the above examples, we can have infinite sets that are finite up-to the group action. We formalise this notion.

\begin{definition}
\label{def:stabiliser}
    For $G\leqslant Aut(\mathbb{A})$ and $x\in\mathcal{V}(\mathbb{A})$, the stabiliser of $x$ in $G$ is
    \begin{equation*}
        G_x :=\left\{\pi\in G: \pi(x)=x\right\}
    \end{equation*}
\end{definition}

\begin{definition}
    For a group $G$ acting on $X$ and $x\in X$, the $G$-orbit of $x$ is $G\cdot \left\{x\right\}=\left\{\pi(x):\pi\in G\right\}$. 
\end{definition}

\begin{definition}[Orbit-Finite]
\label{def:orbit-finite}
    Let $G\leqslant Sym(\mathbb{A})$ and $x\in\mathcal{V}(\mathbb{A})$ be hereditarily finitely supported. We define $x$ to be orbit-finite if there is $\mathbf{a}\in\mathbb{A}^{<\omega}$ s.t. $x$ is a finite union of $G_{\mathbf{a}}$-orbits.
\end{definition}

\begin{example}
\label{examples:orbit-finite}
    \begin{itemize}
        \item $\mathbb{A}$ is a single orbit so is orbit-finite when $G$ is transitive.
        \item For $x$ infinite, $\mathcal{P}(x)$ is not orbit-finite.
        \item For $x$ finite, $x$ is orbit-finite.
    \end{itemize}
\end{example}
We want to say that `enough' sets are orbit-finite.
\begin{definition}[Oligomorphic]
    $G\leqslant Sym(\mathbb{A})$ is oligomorphic if for all $n\in\mathbb{N}$, $\mathbb{A}^n$ is a finite union of $G$-orbits.
    \newline When $G=Aut(\mathcal{A})$, we say that $\mathcal{A}$ is oligomorphic.
\end{definition}

\begin{remark}
    $G$ is oligomorphic iff for all $n\in\mathbb{N}$, $\mathbb{A}^n$ is orbit-finite.
    \newline $\Rightarrow$ trivial
    \newline $\Leftarrow$ for $\mathbf{a}\in\mathbb{A}^{<\omega}$, $G_\mathbf{a}$-orbits refine $G$-orbits so $\mathbb{A}^{n}$ is contained within a finite union of $G$-orbits and $G$ fixes $\mathbb{A}^n$ so $\mathbb{A}^n$ is equal to the finite union of $G$-orbits.
\end{remark}

\begin{example}
\label{examples:oligomorphic}
    \begin{enumerate}
        \item If $\mathcal{A}$ is countable then $Th(\mathcal{A})$ is $\aleph_0$-categorical iff $\mathcal{A}$ is oligomorphic (Ryll-Nardzewski \cite{hodges_model_1993}).
        \begin{itemize}
            \item The structure on the atoms in the language of equality, $\left\langle \mathbb{A}\right\rangle$.
            \item The structure of inequality on the rational numbers, $\left\langle \mathbb{Q},\leq\right\rangle$, models the theory of dense, linear orders (Cantor's Isomorphism Theorem \cite{bhattacharjee_rational_1997}).
            \item The Rado Graph, $\left\langle V, E\right\rangle$, where $V$ is the set of vertices and $E$ is the binary edge relation (Erd\H{o}s and R\'{e}nyi \cite{erdos_asymmetric_1963}).
        \end{itemize}
        \item The group of bijections on $\mathbb{A}$, $Sym(\mathbb{A})$, is oligomorphic as $Sym(\mathbb{A})=Aut(\left\langle \mathbb{A}\right\rangle)$.
        \item Supergroups of oligomorphic groups are oligomorphic.
    \end{enumerate}
\end{example}

\begin{lemma}[\cite{bojanczyk_slightly_nodate}]
    \label{thm:orbit_finite_equiv}
    If $G$ is oligomorphic and $x\in\mathcal{V}(\mathbb{A})$ then 
    \begin{align*}
        &x\text{ is orbit-finite}\\
        \text{iff }&\text{for all }\mathbf{a}\in\mathbb{A}^{<\omega},x\text{ is contained within a finite union of }G_\mathbf{a}\text{-orbits}
    \end{align*}
\end{lemma}

\begin{remark}
    \label{rem:orbit-finite2_implies_orbit-finite1}
    Without $G$ being oligomorphic, we still have that the second condition implies the first. Indeed, suppose that the second condition holds on $x$ then there exists $\mathbf{a}\in\mathbb{A}^{<\omega}$ s.t. $x$ is supported by $\mathbf{a}$ so that $x$ is a union of $G_{\mathbf{a}}$-orbits and $x$ is contained within a finite union of $G_\mathbf{a}$-orbits. So $x$ is a finite union of $G_\mathbf{a}$-orbits as orbits are disjoint.
\end{remark}

\subsection{Definability}
\label{sec:definability}
We can use orbit-finite sets in orbit-finite constructions e.g. Turing machines. We usually do this by specifying a structure on the atoms, $\mathcal{A}$, and using $Aut(\mathcal{A})$ to define orbit-finiteness. However, $Aut(\mathcal{A})$ is not explicitly known. Instead, we will work with the types (in the sense of logic and model-theory) of $\mathcal{A}$ as these are definable. Since we are working with a specific structure, we will abuse notation by conflating the types of $\mathcal{A}$ and their realisations in $\mathcal{A}$.
\begin{definition}
    \label{def:type}
    For $\mathbf{a}\in\mathbb{A}^{n}$ define:
    \begin{equation*}
        tp_\mathcal{A}(\mathbf{a})=\left\{\mathbf{b}\in\mathbb{A}^n:\text{for all formulae }\phi(x_0,\dots,x_{n-1}), \mathcal{A}\models\phi(\mathbf{a})\text{ iff }\mathcal{A}\models\phi(\mathbf{b})\right\}
    \end{equation*} 
\end{definition}
The types of $\mathcal{A}$ are always unions of $Aut(\mathcal{A})$-orbits but will not necessarily be $Aut(\mathcal{A})$-orbits. Hence, working with $\mathcal{A}$-types may differ from working with $Aut(\mathcal{A})$-orbits, suggesting the following definition.
\begin{definition}
\label{def:aleph_0_homo}
    Let $\mathcal{A}$ be a structure on $\mathbb{A}$. Define $\mathcal{D}$ to be $\aleph_0$-homogeneous if for all $n\in\mathbb{N}$ and $\mathbf{a},\mathbf{b}\in\mathbb{A}^n$ with $tp_\mathcal{A}(\mathbf{a})=tp_\mathcal{A}(\mathbf{b})$, there is $\pi\in Aut(\mathcal{A})$ with $\pi(\mathbf{a})=\mathbf{b}$ i.e. the orbits of $\mathbb{A}^{<\omega}$ are the $\mathcal{A}$-types.
\end{definition}
The following standard result from model theory provides numerous $\aleph_0$-homogeneous structures.
\begin{theorem}[\cite{hodges_model_1993}]
    If $\mathcal{A}$ is a countable and oligomorphic structure, then $\mathcal{A}$ is $\aleph_0$-homogeneous.
\end{theorem}
\begin{example}
\label{examples:aleph_0_homo}
The following structures are $\aleph_0$-homogeneous: $\left\langle \mathbb{A}\right\rangle$; $\left\langle \mathbb{Q},\leq\right\rangle$; $\left\langle \mathbb{N},\underline{}+1\right\rangle$; $\left\langle \mathbb{R},\leq\right\rangle$.
\end{example}
We will use the following construction of partial orders in examples throughout this thesis.
\begin{definition}
    \label{def:order_prod_coprod}
    If $\left\langle A,\leq_A\right\rangle$ and $\left\langle B,\leq_B\right\rangle$ are partial orders then define $\left\langle A\oplus B,\leq_{A\oplus B}\right\rangle$ to be the partial order on the disjoint union of $A$ and $B$, $A\oplus B$, where all elements of $A$ are smaller than all elements of $B$ and define $\left\langle A\otimes B,\leq_{A\otimes B}\right\rangle$ to be the partial order on the product of $A$ and $B$, $A\otimes B$, given by the lexicographic order. These are coproducts and products in the category of partial orders. 
\end{definition}
\begin{example}
The structure $\left\langle \mathbb{R}\oplus \mathbb{Q},\leq\right\rangle$ is not $\aleph_0$-homogeneous but is oligomorphic. See Example \ref{examples:D}.
\end{example}

Our current goal is to form orbit-finite sets from $\mathcal{A}$-types, which we do in the following theorems.

\begin{theorem}
\label{thm:equivariant_union_of_types}
    If $\mathcal{A}$ is $\aleph_0$-homogeneous, then $X\subseteq\mathbb{A}^n$ is equivariant iff it is a union of $\mathcal{A}$-types.
\end{theorem}

\begin{proof}
    Let $X\subseteq\mathbb{A}^n$.
    \newline Suppose that $X$ is equivariant and let $\mathbf{a}\in X$. For $\mathbf{b}\in tp_\mathcal{A}(\mathbf{a})$ there is $\pi\in Aut(\mathcal{A})$ with $\pi(\mathbf{a})=\mathbf{b}$, as $\mathcal{A}$ is $\aleph_0$-homogeneous, so $\mathbf{b}\in X$, as $X$ is equivariant. So $\mathbf{a}\in tp_\mathcal{A}(\mathbf{a})\subseteq X$ for $\mathbf{a}\in X$. So $X=\bigcup_{\mathbf{a}\in X} {tp}_{\mathcal{A}}(\mathbf{a})$.
    \newline Suppose that $X$ is a union of $\mathcal{A}$-types and let $\pi\in Aut(\mathcal{A})$. For $\mathbf{a}\in X$, $X\supseteq tp_\mathcal{A}(\mathbf{a})=tp_\mathcal{A}(\pi(\mathbf{a}))=tp_\mathcal{A}(\pi^{-1}(\mathbf{a}))$ so $\pi(\mathbf{a}),\pi^{-1}(\mathbf{a})\in X$. So $\pi(X)\subseteq X=\pi^{-1}(\pi(X))\subseteq\pi^{-1}(X)\subseteq X$ hence $\pi(X)=X$.
\end{proof}

\begin{corollary}
\label{cor:supported_union_of_types}
    If $\mathcal{A}$ is $\aleph_0$-homogeneous, then $X\subseteq \mathbb{A}^k$ is supported by $\mathbf{a}\in\mathbb{A}^n$ iff $X$ is a union of $\mathcal{A}$-types in parameters $\mathbf{a}$.
\end{corollary}

\begin{proof}
    Right to left is trivial.
    \newline Suppose that $\mathbf{a}\in\mathbb{A}^n$ supports $X\subseteq\mathbb{A}^k$. Let
    \begin{equation*}
        Y:=\left\{\pi(\mathbf{a}^\frown \mathbf{b}):\mathbf{b}\in X,\pi\in Aut(\mathcal{A})\right\}
    \end{equation*}
    $Y$ is equivariant so, by Theorem \ref{thm:equivariant_union_of_types}, $Y=\bigcup_{\mathbf{c}\in Y} {tp}_{\mathcal{A}}(\mathbf{c})$.
    \newline Note that $X=\left\{\mathbf{b}\in\mathbb{A}^k:\mathbf{a}^\frown\mathbf{b}\in Y\right\}$ to conclude that $X=\bigcup_{\mathbf{a}^\frown\mathbf{b}\in Y} {tp}_{\mathcal{A}\cup\mathbf{a}}(\mathbf{b})$.
\end{proof}

\begin{remark}
    The converse is also true. If $\mathcal{A}$ is not $\aleph_0$-homogeneous and $\mathbf{a}\in\mathbb{A}^n$ is a witness (there is $\mathbf{b}\in tp_\mathcal{A}(\mathbf{a})$ with $\mathbf{b}\notin Aut(\mathcal{A})\cdot\left\{\mathbf{a}\right\}$), then $Aut(\mathcal{A})\cdot\left\{\mathbf{a}\right\}$ is equivariant but not a union of $\mathcal{A}$-types.
\end{remark}

%% file: text/sec3-permutation-models.tex
\section{Permutation Models}
\label{sec:permutation_models}

The class of hereditarily finitely supported elements models $ZFA$.
\newline This is Fraenkel's construction.
\newline This construction generalises to `permutation models', and the `finite-support permutation models' are a special case.

\subsection{Summary}
In this section we:
\begin{itemize}
    \item Introduce permutation models of $ZFA$.
    \begin{itemize}
        \item Construct the `first Fraenkel model'.
    \end{itemize}
    \item Show that, for a fixed first-order structure $\mathcal{A}$, extending a universe of $ZF$ by atoms in bijection with $\mathcal{A}$ and taking a finite-support permutation model of a structure on the atoms induced by such a bijection, is unique up-to isomorphism (Remark \ref{rem:pm_well-defined}).
    \item We characterise automorphism groups of structures of arity ${<}\kappa$ as closed groups in some topology, for $\kappa$ an infinite cardinal (Theorem \ref{thm:closed_group}).
    \begin{itemize}
        \item This generalises the standard result for standard automorphism groups as closed groups in the topology of pointwise convergence.
    \end{itemize}
    \item We show that taking closures in the above topology preserves hereditary ${<}\kappa$-supports (Theorem \ref{thm:support_closure}).
    \item We prove that all ${<}\kappa$-support permutation models are induced by the automorphism group of some structure of arity ${<}\kappa$ (Corollary \ref{thm:adequacy_automorphism}).
    \begin{itemize}
        \item This justifies the focus on structures in the literature, as using arbitrary groups.
    \end{itemize}
    \item We show that structures on the atoms induce the same permutation model if and only if they are mutually symmetric (Theorem \ref{thm:different_structure_model}).
\end{itemize}

\subsection{Definitions }
We can rephrase being finitely supported (with regards to $G$) using subgroups of $G$.
For all $x$, $x$ is finitely supported iff there exists $ \mathbf{a}\in\mathbb{A}^{<\omega}$ with $G_\mathbf{a}\leqslant G_x$.
\begin{definition}
    \label{def:finite_support_filter}
    Define the finite-support filter on $G$ to be:
    \begin{equation}
        \mathcal{F}^G_{fin}:=\left\{H\leqslant G_{\mathbf{a}}:\mathbf{a}\in\mathbb{A}^{<\omega} \right\}
    \end{equation}
\end{definition}
Observe that for all $x$, $x$ is finitely supported iff $G_x\in\mathcal{F}^G_{fin}$.
Note that $\mathcal{F}$ is a filter of groups on $G$. Indeed, being a filter is almost sufficient for the construction of permutation models to work.

\begin{definition}[Normal Filter \cite{mostowski_uber_1939}]
\label{def:normal_filter}
    Let $\mathbb{A}$ be the set of atoms and $G\leqslant Sym(\mathbb{A})$.
    \newline $\mathcal{F}\subseteq\mathcal{P}(G)$ is a normal filter of subgroups of $G$ if:
    \begin{itemize}
        \item for $H\in\mathcal{F}$, $H$ is a subgroup of $G$;
        \item for $H\in\mathcal{F}$, if $H\leqslant H'\leqslant G$ then $H'\in\mathcal{F}$;
        \item $\mathcal{F}$ is non-empty;
        \item for $H,H'\in\mathcal{F}$, $H\cap H'\in\mathcal{F}$;
        \item for $H\in\mathcal{F}$ and $\pi\in G$, $\pi^{-1} H \pi\in\mathcal{F}$; and
        \item for $a\in\mathbb{A}$, $G_a\in\mathcal{F}$.
    \end{itemize}
\end{definition}

\begin{example}
\label{examples:filters}
    Let $G\leqslant Sym(\mathbb{A})$.
    \begin{itemize}
        \item The finite support filter, $\mathcal{F}^G_{fin}$, is the upward closure of $\left\{G_{\mathbf{a}}:\mathbf{a}\in\mathbb{A}^{<\omega}\right\}$. This is minimal.
        \item The countable support filter, the upward closure of $\left\{G_f:f\in\mathbb{A}^\mathbb{N}\right\}$.
        \item The trivial filter $\left\{H\leqslant G\right\}$.
        \item For $\mathcal{F}$, a normal filter on $G$, and $H\in\mathcal{F}$, $H\#\mathcal{F}:=\left\{H\cap K:K\in\mathcal{F}\right\}$ is a normal filter on $H$.
    \end{itemize}
\end{example}
Two of the above examples are generated by filter bases.
\begin{definition}
\label{def:filter_basis}
    $\mathcal{B}$, a collection of subgroups of $G$, is a filter basis for a filter $\mathcal{F}$ on $G$ iff $\mathcal{F}=\left\{H\leqslant G:\exists B\in\mathcal{B}.B\leqslant H\right\}$.
\end{definition}

\begin{remark}
\label{rem:filter_basis}
    A collection $\mathcal{B}$ of subgroups of $G$ is a filter basis for some filter on $G$ (its upward closure, $\left\langle\mathcal{B}\right\rangle$) iff for all $a\in\mathbb{A}$ there exists $B\in \mathcal{B}$ with $B\leqslant G_a$ and for all $B,B'\in\mathcal{B}$ there exists $B''\in\mathcal{B}$ with $B''\leqslant B\cap B'$.
\end{remark}

\begin{example}
\label{examples:filter_basis_from_supports}
    Let $S$ be a set of subgroups of $G$ closed under conjugation by $G$.
    Let
    \begin{equation*}
        \# S:=\left\{\bigcap \mathcal{S}:\mathcal{S}\subseteq_{fin} S\right\}
    \end{equation*}
    If for all $a\in\mathbb{A}$ there exists $H\in\#S$ with $H\leqslant G_a$ then $\#S$ is a filter basis.
\end{example}

\begin{definition}
\label{def:kappa-support_filter}
    For $\kappa$ an infinite cardinal, the $\kappa$-support filter on $G$ is:
    \begin{equation*}
        \mathcal{F}^G_{{<}\kappa}:=\left\langle\left\{G_\mathbf{a}:\mathbf{a}\in\mathbb{A}^{{<}\kappa}\right\}\right\rangle
    \end{equation*}
    The finite-support filter is the case $\mathcal{F}^G_{<\aleph_0}=\mathcal{F}^G_{fin}$.
\end{definition}

\begin{definition}[Symmetric]
\label{def:symmetric}
    Let $x\in\mathcal{V}(\mathbb{A})$, $G\leqslant Sym(\mathbb{A})$, and $\mathcal{F}$ be a normal filter on $G$.
    \begin{itemize}
        \item $x$ is symmetric (with respect to $\mathcal{F}$) iff $G_x\in\mathcal{F}$.
        \item $x$ is hereditarily symmetric (with respect to $\mathcal{F}$) iff for all $y\in TC(\{x\})$, $y$ is symmetric.
    \end{itemize}
\end{definition}
\begin{definition}[Permutation Models]
\label{def:permutation_models}
    Let:
    \begin{itemize}
        \item $\left\langle\mathcal{M},\mathbb{A},\in\right\rangle$ be a model of $ZFA$; and
        \item $G,\mathcal{F}\in\mathcal{M}$ be a group and normal filter on the atoms;
    \end{itemize}
    then define the class:
    \begin{equation*}
        PM(\mathcal{M},G,\mathcal{F}):=\left\{x\in\mathcal{M}:x\text{ is hereditarily symmetric with respect to }\mathcal{F}\right\}
    \end{equation*}
\end{definition}

\begin{remark}
\label{rem:truth_relative_models}
    The elements of different models of $ZFA$ are different so truth in these models differ. In Definition \ref{def:permutation_models}, we require $\mathcal{M}$ to model that $G$ is a group and $\mathcal{F}$ is a normal filter on the atoms. In this case, the model, where the assumptions must hold, is obvious but this detail will become important as we work with multiple models of $ZFA$. We shall be explicit about this detail whenever it's ambiguous.
    For example, the definition of $PM(\mathcal{M},G,\mathcal{F})$ is inside $\mathcal{M}$. To be explicit, we could say:
    \begin{equation*}
        PM(\mathcal{M},G,\mathcal{F}):=\left\{x\in\mathcal{M}:\mathcal{M}\models x\text{ is hereditarily symmetric with respect to }\mathcal{F}\right\}
    \end{equation*}
\end{remark}

\begin{definition}
\label{def:natural_filter}
    Let:
    \begin{itemize}
        \item $\left\langle\mathcal{M},\mathbb{A},\in\right\rangle\models ZFA$;
        \item $G\leqslant Sym(\mathbb{A})$; and
        \item $\mathcal{F}$ be a normal filter on $G$;
    \end{itemize}
    define
    \begin{equation*}
        \mathcal{F}_{nat}:=\left\langle\left\{G_x|x\in PM(\mathcal{M},G,\mathcal{F})\right\}\right\rangle
    \end{equation*}
\end{definition}

\begin{remark}
    \label{rem:natural_filter}$PM(\mathcal{M},G,\mathcal{F})=PM\left(\mathcal{M},G,\mathcal{F}_{nat}\right)$ as they contain the same stabilisers.
\end{remark}
\begin{example}
    \label{examples:natural_filter}
    All filters of the form in Example \ref{examples:filter_basis_from_supports}, $\mathcal{F}$, are natural filters i.e. if $\mathcal{F}_{nat}=\mathcal{F}$. This includes $\mathcal{F}_{{<}\kappa}$ for all infinite cardinals $\kappa$. If $\mathcal{F}$ is a normal filter on $G$, then ${\left(\mathcal{F}_{nat}\right)}_{nat}=\mathcal{F}_{nat}$ as $\mathcal{F}_{nat}$ is of the form in Example \ref{examples:filter_basis_from_supports}.
\end{example}

\begin{theorem}[\cite{jech_axiom_1973}]
\label{theorem:permutation_model_ZFA}
    For $\left\langle\mathcal{M},\mathbb{A},\in\right\rangle\models ZFA$, $G\in\mathcal{M}$ a group acting on $\mathbb{A}$, and $\mathcal{F}\in\mathcal{M}$ a normal filter on $G$, then $\left\langle PM(\mathcal{M},G,\mathcal{F}),\mathbb{A},\in\right\rangle$ is a model of $ZFA$.
\end{theorem}

\begin{example}
    \label{examples:FFM}
    Let $\left\langle\mathcal{M},\mathbb{A},\in\right\rangle\models ZFA+\textit{`$\mathbb{A}$ is infinite'}$.
    \newline Then $\mathcal{N}=PM\left(\mathcal{M},Sym(\mathbb{A}),\mathcal{F}^{Sym(\mathbb{A})}_{fin}\right)$ is the first Fraenkel Model. 
    \begin{claim}
        $\mathcal{N}\not\models AC$ (the axiom of choice)
    \end{claim}
    \begin{proof}
        $AC$ implies that every set is in bijection with an ordinal, which are pure. So, if $\mathcal{N}\models AC$ then there is a surjection $f:x\twoheadrightarrow\mathbb{A}$ in $\mathcal{N}$ for some pure set $x$. The property of being a surjection from a pure set onto $\mathbb{A}$ is absolute between transitive classes. So, it suffices to check all $f:x\twoheadrightarrow\mathbb{A}\in\mathcal{M}$ to contradict choice in $\mathcal{N}$.
    \newline Let $f:x\twoheadrightarrow\mathbb{A}\in\mathcal{M}$ for $x$ pure. If $f\in\mathcal{N}$ then it is supported by some $\mathbf{a}\in\mathbb{A}^{<\omega}$. Let $b_0\not=b_1\in\mathbb{A}\setminus\mathbf{a}$ and consider $\left(\begin{array}{cc}b_0&b_1\end{array}\right)\in Sym(\mathbb{A})_\mathbf{a}$. $\left(\begin{array}{cc}b_0&b_1\end{array}\right)f\not=f$ as $b_0=f(x_0)$ for some $x_0\in x$ and $b_0\not= b_1=\left(\left(\begin{array}{cc}b_0&b_1\end{array}\right)f\right)(x_0)$. So $f$ is not supported by $\mathbf{a}$ for any $\mathbf{a}\in\mathbb{A}^{<\omega}$ so $f\notin\mathcal{N}$. $\mathcal{N}$ contains no surjection from a pure set to the set of atoms, so $\mathcal{N}\not\models AC$.
    \end{proof}
\end{example}
So, if there is a model of $ZFA$ then there is a model rejecting the axiom of choice.
\newline Now suppose $\mathcal{M}\models ZFA$ and $\mathcal{A^*}\in\mathcal{M}$ is a pure structure with $\mathcal{M}\models\left|\mathbb{A}\right|=\left|\mathcal{A}^*\right|$. We can biject $\mathbb{A}$ with $\mathcal{A}^*$ to construct a structure $\mathcal{A}\cong \mathcal{A}^*$ on the atoms and take the finite-support permutation model $PM(\mathcal{M},Aut(\mathcal{A}),\mathcal{F}^{Aut(\mathcal{A})}_{fin})$.
\begin{remark}
    \label{rem:pm_well-defined}
    The isomorphism $\mathcal{A}\cong \mathcal{A}^*$ induces an isomorphism $\mathcal{M}\cong \mathcal{V}^*(\mathcal{A}^*)$ by Theorem \ref{thm:atoms_iso_pure_class}, which restricts to an isomorphism
    \begin{equation*}
        PM(\mathcal{M},Aut(\mathcal{A}),\mathcal{F}^{Aut(\mathcal{A})}_{fin})\cong PM(\mathcal{V}^*(\mathcal{A}^*),Aut(f(F(\mathcal{A}^*))),\mathcal{F}^{Aut(f(F(\mathcal{A}^*)))}_{fin})
    \end{equation*}
    where the left-hand side is taken relative to $\mathcal{M}$, the right-hand side is taken relative to $\mathcal{V}^*(\mathcal{A}^*)$, $F$ is the map from Theorem \ref{thm:zfa_in_zf}.\ref{thm:zfa_in_zf_2}, and $f$ is the bijection between $F(\mathcal{A}^*)$ and the atoms in $\mathcal{V}^*(\mathcal{A}^*)$ from Theorem \ref{thm:zfa_in_zf}.\ref{thm:zfa_in_zf_3}.
\newline This isomorphism is definable so for every formula, $\phi(x)$, there is a formula, $\phi^*(x)$, with
\begin{equation*}
    \left(\phi(\mathcal{A})\right)^{PM(\mathcal{M},Aut(\mathcal{A}),\mathcal{F}^{Aut(\mathcal{A})}_{fin})}\leftrightarrow\left(\phi^*(\mathcal{A}^*\right))^{\mathcal{V}}
\end{equation*}
\newline Note that we made of choice of bijection between $\mathcal{A}^*$ and $\mathbb{A}$ to induce $\mathcal{A}$ and we could have chosen a different bijection to induce $\mathcal{A'}$ on the atoms instead. The right-hand-side is independent of this choice of bijection so that:
\begin{equation*}
    PM(\mathcal{M},Aut(\mathcal{A}),\mathcal{F}^{Aut(\mathcal{A})}_{fin})\cong PM(\mathcal{M},Aut(\mathcal{A}'),\mathcal{F}^{Aut(\mathcal{A}')}_{fin})
\end{equation*}
In other words, this construction is unique up-to-isomorphism.
\newline \textbf{The logic of a finite-support permutation model induced by a pure structure is fully determined by the pure universe}
\begin{definition}
\label{def:permutation_model_from_structure}
    We will write $PM(\mathcal{A}^*)$ for the finite-support permutation model induced in $\mathcal{V}^*(\mathcal{A}^*)$ by the structure $\underline{\mathcal{A}^*}:=f(F(\mathcal{A}^*))$. This determines, up-to-isomorphism, the above construction of $PM(\mathcal{M},Aut(\mathcal{A}),\mathcal{F}^{Aut(\mathcal{A})}_{fin})$. For $a\in\mathcal{A}^*$, we write $\underline{a}$ for the image of $a$ under this mapping.
\end{definition}

\end{remark}

\subsection{Adequacy of Automorphism Groups}

Section \ref{sec:definability} only applies when using the automorphism group of some first-order structure.
\newline Taking permutation models is more general than this. However, we will see that every finite-support permutation model arises from an automorphism group.

\begin{definition}[Topology of Point-wise Convergence on $\mathbb{A}^{\mathbb{A}}$]
\label{def:topology_of_point-wise_convergence}
    Define a topology on $\mathbb{A}^{\mathbb{A}}$ by giving $\mathbb{A}$ the discrete topology and taking the product topology.
    \newline The basic open sets are $U_s:=\left\{f\in\mathbb{A}^\mathbb{A}:s\subseteq f\right\}$, where $s:\mathbb{A}\rightharpoonup_{fin}\mathbb{A}$.
    \newline This induces the subspace topology on $Sym(\mathbb{A})$ with basis $S_s:=U_s\cap Sym(\mathbb{A})$.
\end{definition}
Although we mostly care about finite-support permutation models, the theorems in this section generalise to ${<}\kappa$-support permutation models (for $\kappa$ a cardinal). We will prove results for general $\kappa$.

\begin{definition}
\label{def:kappa_product_topology}
    Let $\kappa$ be an infinite cardinal. We define the ${<}\kappa$-product topology on $\mathbb{A}^\mathbb{A}$ by giving the basis $\left\{\bigcap_{j\in J}\pi^{-1}_j\left(U_j\right):J\subseteq\mathbb{A},\left|J\right|{<}\kappa , \forall j\in J. U_j\subseteq\mathbb{A}\right\}$.
    \newline This induces the subspace topology on $Sym(\mathbb{A})$.
    \newline Note that the ${<}\aleph_0$-product topology is the topology of point-wise convergence.
    \newline As before, we can say that the basic open sets are of the form $U_s:=\left\{f\in\mathbb{A}^\mathbb{A}:s\subseteq f\right\}$ for $s:\mathbb{A}\rightharpoonup_{{<}\kappa}\mathbb{A}$, inducing basic open sets on $Sym(\mathbb{A})$, $S_s:=U_s\cap Sym(\mathbb{A})$.
\end{definition}

\begin{definition}
\label{def:kappa_structure}
    Let $\kappa$ be an infinite cardinal. We define $\kappa$-structures to be like standard first-order logical structures, except we allow predicates of all arity strictly less than $\kappa$.
    \newline Note that $\aleph_0$-structures are the structures in first-order logic.
\end{definition}

\begin{theorem}[Closed in $Sym(\mathbb{A})$]
\label{thm:closed_group}
Let $\kappa$ be an infinite cardinal.
\newline For $G\leqslant Sym(\mathbb{A})$, $G$ is closed in the ${<}\kappa$-topology iff $G$ is an automorphism group of a ${<}\kappa$-structure.
The case $\kappa=\aleph_0$ is a standard result.
\end{theorem}
\begin{proof}
Let $G\leqslant Sym(\mathbb{A})$ be closed in the ${<}\kappa$-topology. Define:
    \begin{align*}
        P:=&\bigcup_{n{<}\kappa} \left\{E\subseteq\mathbb{A}^n:E\text{ is a }G\text{-orbit}\right\}\\
        \mathcal{A}:=&\left\langle\mathbb{A},P\right\rangle\\
    \end{align*}
    Clearly, $G\leqslant Aut(\mathcal{A})$ as $G$ fixes all the predicates.
    Let $\pi\in Aut(\mathcal{A})$.
    \begin{claim}
        $\pi$ is a limit point of $G$ and hence $\pi\in G$.
    \end{claim}
    Let $s:\mathbb{A}\rightharpoonup\mathbb{A}$ be a partial function with domain of cardinality $n=\left|s\right|<\kappa$ s.t. $\pi\in S_s$. Let $\mathbf{a}\in\mathbb{A}^n$ enumerate $dom\left(s\right)$. $\pi\left(\mathbf{a}\right)
        \in G\cdot\left\{\mathbf{a}\right\}$ as $\pi$ fixes the $G$-orbits, so there exists $\sigma\in G$ with $\sigma\left(\mathbf{a}\right)=\pi\left(\mathbf{a}\right)$. Hence, $\sigma\in S_s$. $G$ intersects all the basic open sets containing $\pi$ so, as $G$ is closed, $\pi\in G$.
    \begin{remark}
        If $G$ is oligomorphic, $\kappa=\aleph_0$ (and a countable union of countable sets is countable, which is a consequence of $AC$) then for all $n\in\mathbb{N}$, $\mathbb{A}^n$ has finitely many $G$-orbits, so $P$ is countable.
    \end{remark}
    We continue with the other direction of the proof of Theorem \ref{thm:closed_group}.
    \newline Let $\mathcal{A}$ be a $\kappa$-structure on $\mathbb{A}$ and let $\pi\in Sym(\mathbb{A})\setminus Aut(\mathcal{A})$. Then there exists a predicate, $P$, of $\mathcal{A}$ with $\mathbf{a}\in P$ s.t. $\pi(\mathbf{a})\notin P$. Let $S:=S_{\pi\restriction\mathbf{a}}$. Note that $\pi\in S$.
    For $\sigma\in S$, $\sigma(P)\not= P$, as $\pi(\mathbf{a})\in \sigma(P)$, so $\sigma\notin Aut(\mathcal{A})$. Now for $\pi\in Sym(\mathbb{A})\setminus Aut(\mathcal{A})$, there is an open set $S$ with $\pi\in S\subseteq Sym(\mathbb{A})\setminus Aut(\mathcal{A})$. Hence $Sym(\mathbb{A})\setminus Aut(\mathcal{A})$ is open and $Aut(\mathcal{A})$ is closed.
\end{proof}
\begin{remark}
    The structure constructed above is $\aleph_0$-homogeneous and can be constructed from an existing automorphism group. So every structure can be extended to a $\aleph_0$-homogeneous one, without changing the automorphism group.
\end{remark}
This suggests that we can construct a candidate automorphism group by taking the closure of a group; we must check this that will be a group.
\begin{claim}
    If $G\leqslant Sym(\mathbb{A})$ is a group, then so is $\overline{G}$, the closure of $G$ in $Sym(\mathbb{A})$.
\end{claim}
\begin{proof}
    First note that $G\subseteq\overline{G}\subseteq Sym(\mathbb{A})$ so $\overline{G}$ contains the identity and consists of bijections. We show that $\overline{G}$ is closed under composition and inverses. 
    \newline Let $\pi,\sigma\in\overline{G}$. Let $S\subseteq_{{<}\kappa}\mathbb{A}$ so, by definition of the topology, there are $\pi',\sigma'\in G$ s.t. for all $a\in S$, $\pi(a)=\pi'(a)$ and $\sigma(\pi(a))=\sigma'(\pi(a))$. $\sigma'\circ\pi'\in G$ and $\left(\sigma'\circ\pi'\right)\restriction S=\left(\sigma\circ\pi\right)$. So $G$ intersects all basic open sets containing $\left(\sigma\circ\pi\right)$ hence $\left(\sigma\circ\pi\right)\in \overline{G}$.
    \newline Let $\pi\in G$. Let $S\subseteq_{{<}\kappa}\mathbb{A}$ so there is $\pi'\in G$ with $\pi'\restriction \pi^{-1}(S)=\pi\restriction \pi^{-1}(S)$. $\pi'^{-1}\in G$ and $\pi'^{-1}\restriction S =\pi^{-1}\restriction S$. So $G$ intersects all basic open sets containing $\pi^{-1}$ hence $\pi^{-1}\in\overline{G}$.
\end{proof}

Now, $\overline{G}$ is the smallest automorphism group (of a $\kappa$-structure) containing $G$.
Recall from Section \ref{sec:definability} that working with automorphism groups is preferable, since it allows us to work with types. This motivates our goal, stated at the start of this section, to create an automorphism group that generates the same finite-support permutation model as $G$. We have a stronger result: the finite-supports themselves are preserved. Again, this generalises to infinite cardinalities $\kappa$.

\begin{theorem}[Supports are Preserved under Closure]
    \label{thm:support_closure}
    Let $\kappa$ be an infinite cardinal and consider the ${<}\kappa$-topology on $Sym(\mathbb{A})$.
    \newline Let  $G\leqslant Sym(\mathbb{A})$ and let $x$ be an element of the universe then for all $\mathbf{a}\in\mathbb{A}^{{<}\kappa}$:
    \begin{align*}
        &\text{$x$ is hereditarily ${<}\kappa$-supported and has support $\mathbf{a}$ with respect to $G$}\\
        \text{iff }&\text{$x$ is hereditarily ${<}\kappa$-supported and has support $\mathbf{a}$ with respect to $\overline{G}$}
    \end{align*}
\end{theorem}
\begin{proof}
    $\left(\Leftarrow\right): G\leqslant\overline{G}$ so the claim is trivial.
    \newline $\left(\Rightarrow\right):$ We prove by $\in$-induction.
    \newline Let $a\in\mathbb{A}$ and suppose that $\mathbf{a}\in\mathbb{A}^{{<}\kappa}$ supports $a$ with respect to $G$. For $\pi\in \left(\overline{G}\right)_\mathbf{a}$ there is $\sigma\in G$ s.t. $\sigma(\mathbf{a})=\mathbf{a}$ and $\sigma(a)=\pi(a)$, by the topology. $\sigma(a)=a$ as $\mathbf{a}$ supports $a$ with respect to $G$ so $\pi(a)=a$. So $\mathbf{a}$ supports $a$ with respect to $\overline{G}$.
    \newline Suppose that $x$ is a set s.t. for all $y\in TC(x)$, for $\mathbf{a}\in\mathbb{A}^{<\kappa}$ if $y$ is hereditarily ${<\kappa}$-supported and has support $\mathbf{a}$ with respect to $G$ then the same is true with respect to $\overline{G}$. Suppose that $x$ is hereditarily ${<\kappa}$-supported and has support $\mathbf{a}\in\mathbb{A}^{<\kappa}$ with respect to $G$. By the inductive assumption, the elements of $TC(x)$ are ${<}\kappa$-supported with respect to $\overline{G}$  so it suffices to show that $x$ is supported by $\mathbf{a}$ with respect to $\overline{G}$. Let $\pi\in \left(\overline{G}\right)_{\mathbf{a}}$ and $y\in x$. $y$ is supported by some $\mathbf{b}\in\mathbb{A}^{{<}\kappa}$ with respect to $\overline{G}$. By the topology, there exists $\sigma\in G$ s.t. $\sigma(\mathbf{a})=\pi(\mathbf{a})$ and $\sigma(\mathbf{b})=\pi(\mathbf{b})$, since $\left|\mathbf{a}\right|+\left|\mathbf{b}\right|<\kappa+\kappa=\kappa$ (note that this requires $AC$ for $\kappa\not=\aleph_0$ but does not require $AC$ for $\kappa=\aleph_0$). $\sigma\in \overline{G}$ so $\sigma^{-1}\circ \pi \in \overline{G}_{\mathbf{a}\cup\mathbf{b}}$ hence $\left(\sigma^{-1}\circ\pi\right)(y)=y$. It follows that $\pi(y)=\sigma(y)\in x$. For all $y\in x$, $\pi(y)\in x$ so $\pi(x)\subseteq x$. We also know that $\pi^{-1}\in\left(\overline{G}\right)_{\mathbf{a}}$ so $\pi^{-1}(x)\subseteq x$ so $x\subseteq\pi(x)$ thus $\pi(x)=x$. So $x$ is supported by $\mathbf{a}$ with respect to $\overline{G}$.
    \newline Now by $\in$-induction, the claim holds for all elements of the universe.
\end{proof}

\begin{corollary}
\label{cor:permutation_model_group_closure}
    If:
    \begin{itemize}
        \item $\left\langle\mathcal{M},\mathbb{A},\in\right\rangle$ is a model of $ZFA$;
        \item $\kappa$ is an infinite cardinal; and
        \item $G\leqslant Sym(\mathbb{A})$;
    \end{itemize}
    then, considering the ${<}\kappa$-topology,
    \begin{equation}
        PM(\mathcal{M},G,\mathcal{F}^G_{{<}\kappa})=PM(\mathcal{M},\overline{G},\mathcal{F}^{\overline{G}}_{{<}\kappa})
    \end{equation}
\end{corollary}

\begin{proof}
This is immediate from Theorem \ref{thm:support_closure}.
    \begin{align*}
        PM(\mathcal{M},G,\mathcal{F}^G_{{<}\kappa})=&\left\{x\in\mathcal{M}:x\text{ is hereditarily ${<}\kappa$-supported w.r.t. }G\right\}\\
        =&\left\{x\in\mathcal{M}:x\text{ is hereditarily ${<}\kappa$-supported w.r.t. }\overline{G}\right\}\\
        =&PM(\mathcal{M},\overline{G},\mathcal{F}^{\overline{G}}_{{<}\kappa})
    \end{align*}
\end{proof}

If we combine the above work, we now know that any ${<}\kappa$-support permutation model is induced by some $\kappa$-structure.

\begin{corollary}[Adequacy of Automorphism Groups]
    \label{thm:adequacy_automorphism} If $\mathcal{N}$ is a ${<}\kappa$-support permutation model of $\mathcal{M}$, then it is induced by the automorphism group of some $\kappa$-structure.
\end{corollary}

\begin{proof}
Say that
\begin{equation*}
    \mathcal{N}=PM(\mathcal{M},G,{\mathcal{F}}^{G}_{{<}\kappa})
\end{equation*}
By Corollary \ref{cor:permutation_model_group_closure},
\begin{equation*}
    \mathcal{N}=PM(\mathcal{M},\overline{G},{\mathcal{F}}^{\overline{G}}_{{<}\kappa})
\end{equation*}
By Theorem \ref{thm:closed_group},
\begin{equation*}
    \overline{G}=Aut(\mathbb{A},\bigcup_{n{<}\kappa}\left\{E\subseteq\mathbb{A}^n:E\text{ is a }G\text{-orbit}\right\})
\end{equation*}
    Note that if $G$ is oligomorphic, $\kappa=\aleph_0$, and a countable union of finite sets is countable, then the structure has countable language.
\end{proof}

\begin{example}
\label{examples:perm_group}
Let
\begin{equation*}
    Perm(\mathbb{A}):=\left\{\pi\in Sym(\mathbb{A})|\pi\text{ fixes co-finitely many elements of }\mathbb{A}\right\}
\end{equation*}
It is the case that $Perm(\mathbb{A})\leqslant Sym(\mathbb{A})$ and that $\overline{Perm(\mathbb{A})}=Sym(\mathbb{A})$ in the ${<}\aleph_0$-topology on $Sym(\mathbb{A})$ so that they have the same finite-support permutation models. This justifies Pitts \cite{pitts_nominal_2013} use of $Perm(\mathbb{A})$ instead of $Sym(\mathbb{A})$.
\end{example}
\begin{example}
    \label{examples:dense_subgroup_of_Q}
    Let
    \begin{equation*}
        H:=\left\{f\in\mathbb{Q^Q}|f \text{ is a piecewise linear increasing bijection}\right\}
    \end{equation*}
    It is the case that $H\leqslant Aut(\left\langle\mathcal{\mathbb{Q},\leq}\right\rangle)$ and that $\overline{H}=\left\langle\mathcal{\mathbb{Q},\leq}\right\rangle$ in the ${<}\aleph_0$-topology on $Sym(\mathbb{A})$.
\end{example}
The main works on orbit-finite sets in computer science, \cite{pitts_nominal_2013} and \cite{bojanczyk_slightly_nodate}, specialise to the case of finite-support permutation models induced by structures on countable atoms. The finite-support condition is nice, it gives some hope for definability and ensures sufficient dependence between the group and the permutation model (otherwise we could use the symmetry group and have the filter encode the subgroup information), but the additional condition that the group is an automorphism group is more unnatural and suggests a possible generalisation. Corollary \ref{thm:adequacy_automorphism} shows that, specialising to finite-support filters, inducing permutation models from structures is the correct generalisation.

\subsection{Structures With the Same Permutation Models}

It is reasonable to ask when two structures induce the same finite-support permutation model. We generalise to the case of a set of supports $S$, which specialises to the case of $S=\mathbb{A}$ for finite-supports.
\begin{definition}
\label{def:set_of_supports}
    Recall $\#$ from Example \ref{examples:filter_basis_from_supports}. 
    For a set $S$ and $G\leqslant Sym(\mathbb{A})$, define
    \begin{equation*}
        \mathcal{F}^G_S:=\left\langle \#\left\{G_x:x\in S\right\}\right\rangle
    \end{equation*}
    if $\mathcal{F}^G_S$ satisfies the definition of a normal filter from Definition \ref{def:normal_filter}. 
    \newline Fix $S$. If $\mathcal{F}^G_S$ is defined for all $G\leqslant Sym(\mathbb{A})$, we say that $S$ is a set of supports.
\end{definition}
\begin{example}
    $\mathbb{A}^{{<}\kappa}$ is a set of supports for all cardinals $\kappa$.
\end{example}
\begin{remark}
    $\mathcal{F}^G_S=\mathcal{F}^G_{S^{<\omega}}$ as $G_\mathbf{a}=\bigcap_{a\in \mathbf{a}}G_a$.
\end{remark}

\begin{remark}
    \begin{itemize}
        \item $\mathcal{F}^G_\mathbb{A}=\mathcal{F}^G_{\mathbb{A}^{<\omega}}=\mathcal{F}^G_{fin}$.
        \item For $\kappa$ a cardinal, $\mathcal{F}^G_{\mathbb{A}^{{<}\kappa}}=\mathcal{F}^G_{{<}\kappa}$.
    \end{itemize}
\end{remark}

\begin{theorem}
    \label{thm:different_structure_model}
    Let $\left\langle\mathcal{M},\mathbb{A},\in\right\rangle\models ZFA$ and $S\in\mathcal{M}$ be a set of supports.
    \newline If:
    \begin{itemize}
        \item $\mathcal{A},\mathcal{A'}$ are both first-order structures on $\mathbb{A}$ with $\mathcal{A}$ having pure language;
        \item $\mathcal{N}:=PM(\mathcal{M},Aut(\mathcal{A}),\mathcal{F}^{Aut(\mathcal{A})}_S)$;
        \item $\mathcal{N}':=PM(\mathcal{M},Aut(\mathcal{A}'),\mathcal{F}^{Aut(\mathcal{A}')}_S)$; and
        \item $\mathcal{A}\in\mathcal{N}'$;
    \end{itemize}
    then $\mathcal{N}\subseteq\mathcal{N}'$.
\end{theorem}

\begin{proof}
    $\mathcal{A}\in\mathcal{N}'$ so let $\mathbf{a}\in S^{<\omega}$ support $\mathcal{A}$ with respect to $\mathcal{A}'$ (Definition \ref{def:support_structure}).
    \newline $\mathcal{A}$ has language in bijection with a pure set, so the language of $\mathcal{A}$ is equivariant with respect to $\mathcal{A}$. $\mathcal{A}$ can be seen as a function from its language to its predicates. Both $\mathcal{A}$ and the elements of its language are fixed by $Aut(\mathcal{A}')_{\mathbf{a}}$, so the range of $\mathcal{A}$, its predicates, must be fixed by $Aut(\mathcal{A}')_{\mathbf{a}}$.
    \newline Let $x\in\mathcal{M}$, if $x$ is supported by $\mathbf{b}\in S^{<\omega}$ with respect to $\mathcal{A}$, then $x$ is supported by $\mathbf{a}^\frown \mathbf{b}$ with respect to $\mathcal{A}'$. This is since for $\pi\in Aut(\mathcal{A}')_{\mathbf{a}^\frown \mathbf{b}}$, $\pi$ fixes the predicates of $\mathcal{A}$ so $\pi\in Aut(\mathcal{A})_{\mathbf{b}}$ so $\pi(x)=x$.
    \newline So $\mathcal{N}\subseteq\mathcal{N}'$.
\end{proof}

\begin{corollary}
    \label{thm:structure_perm_model_eq}
    Let $\left\langle\mathcal{M},\mathbb{A},\in\right\rangle\models ZFA$ and $S\in\mathcal{M}$ be a set of supports.
    \newline If:
    \begin{itemize}
        \item $\mathcal{A},\mathcal{A'}$ are both first-order structures on $\mathbb{A}$ with pure languages;
        \item $\mathcal{N}:=PM(\mathcal{M},Aut(\mathcal{A}),\mathcal{F}^{Aut(\mathcal{A})}_S)$; and
        \item $\mathcal{N}':=PM(\mathcal{M},Aut\left(\mathcal{A}'\right),\mathcal{F}^{Aut\left(\mathcal{A}'\right)}_S)$;
    \end{itemize}
    then
    \begin{align*}
        &\mathcal{A}\in\mathcal{N}'\text{ and }\mathcal{A}'\in\mathcal{N}\\
        \text{iff }&\mathcal{N}=\mathcal{N}'
    \end{align*}
\end{corollary}
\begin{proof}
    Suppose that $\mathcal{A}\in\mathcal{N}'\text{ and }\mathcal{A}'\in\mathcal{N}$, then by Theorem \ref{thm:different_structure_model} we get that $\mathcal{N}\subseteq\mathcal{N}'$ and $\mathcal{N}'\subseteq\mathcal{N}$ so that $\mathcal{N}=\mathcal{N}'$.
    \newline Suppose that $\mathcal{N}=\mathcal{N}'$. For all structures $\mathcal{A}''$ with a pure language, $\mathcal{A}''$ is equivariant with respect to $\mathcal{A}''$ as $Aut(\mathcal{A}'')$ fixes the language and every predicate of $\mathcal{A}''$. This means that $\mathcal{A}\in\mathcal{N}=\mathcal{N}'$ and $\mathcal{A}'\in\mathcal{N}'=\mathcal{N}$.
\end{proof}
\begin{example}
    Work in $\mathcal{V}^*(\mathbb{Q})$ (Definition \ref{def:class_adding_atoms}) \newline Let $\mathcal{A}=\left\langle\underline{\mathbb{Q}},\leq,\underline{0}\right\rangle$ and $\mathcal{A}'=\left\langle\underline{\mathbb{Q}},\leq,\underline{1}\right\rangle$. $\underline{0}$ supports $\mathcal{A}$ with respect to $\mathcal{A}'$ and $\underline{1}$ supports $\mathcal{A}'$ with respect to $\mathcal{A}'$ so $\mathcal{A}\in PM(\mathcal{V}^*(\mathbb{Q}), Aut(\mathcal{A}'),\mathcal{F}^{Aut(\mathcal{A}')}_{fin})$ and $\mathcal{A}'\in PM(\mathcal{V}^*(\mathbb{Q}), Aut(\mathcal{A}),\mathcal{F}^{Aut(\mathcal{A})}_{fin})$.
    \newline By Corollary \ref{thm:structure_perm_model_eq}, $PM(\mathcal{V}^*(\mathbb{Q}), Aut(\mathcal{A}'),\mathcal{F}^{Aut(\mathcal{A}')}_{fin})=PM(\mathcal{V}^*(\mathbb{Q}), Aut(\mathcal{A}),\mathcal{F}^{Aut(\mathcal{A})}_{fin})$. 
    \newline Hence, $PM(\left\langle\mathbb{Q},\leq,0\right\rangle)=PM(\left\langle\mathbb{Q},\leq,1\right\rangle)$ (Definition \ref{def:permutation_model_from_structure}).
\end{example}

%% file: text/sec4-computation.tex
\section{Logic of Permutation Models}
\label{sec:computation}

\subsection{Summary}
In this section we:
\begin{itemize}
    \item State and prove a strengthening of Brunner's \cite{brunner_fraenkel-mostowski_1990} characterisation of finite-support permutation models as a first-order sentence (Theorem \ref{thm:finite_support_fo_improved}).
    \begin{itemize}
        \item We remove all uses of global choice and replace it with a weakening of choice, where necessary.
        \item We explain why this characterisation does not generalise to $<\kappa$-supports.
    \end{itemize}
    \item We show that the logic of finite-support permutation models to define and reason about orbit-finite structures (Section \ref{sec:computation}). Specifically, we show that:
    \begin{itemize}
        \item Orbit-finite constructions can be defined inside finite-support permutation models on oligomorphic groups.
        \item Statements about orbit-finite constructions are expressible inside finite-support permutation models on oligomorphic groups.
        \item We give the example of Ryll-Nardzewski functions being expressible.
    \end{itemize}
    \item We introduce the notion of set-theoretic forcing \cite{jech_axiom_1973}(Section \ref{sec:Almost-homogeneous_forcing}).
    \begin{itemize}
        \item We provide the definitions and theorems necessary to follow Section \ref{sec:cardinal}.
        \item We state Hall's theorem relating forcing to permutation models.
        \item We explain the impact on the expressibility of permutation models.
    \end{itemize}
\end{itemize}

\subsection{External Properties as Internal Properties}

\subsubsection{Finite-Support as a First-Order Property}

Our aim is to move a finite-support permutation model, $\mathcal{N}$, from one external universe to another. Suppose that we can do this, but that we are only guaranteed that $\mathcal{N}$ is a permutation model; we want a way to tell that $\mathcal{N}$ is a finite-support permutation model in this new universe. This is not guaranteed, since the groups on the atoms will have changed.
\newline One way of doing this, is to find some property expressed in the logic of $\mathcal{N}$ that ensures $\mathcal{N}$ is a finite-support permutation model. This means that if both $\mathcal{M}$ and $\mathcal{M}'$ contain $\mathcal{N}$ as a permutation model and $\mathcal{M}$ causes $\mathcal{N}$ to satisfy this property, then $\mathcal{N}$ is a finite-support permutation model in $\mathcal{M}'$. We first state a weakening of $AC$ formulated by Blass \cite{blass_injectivity_1979}.
\begin{definition}[Small Violations of Choice (SVC)]
\label{def:SVC}
    Let $\mathcal{M}\models ZFA$ and $x\in\mathcal{M}$. Then define $SVC(x)$ to hold in $\mathcal{M}$ if for all sets $y\in\mathcal{M}$, there exists an ordinal $\alpha$ and a surjection from $x\times\alpha$ onto $y$.
\end{definition}

\begin{lemma}
    \label{thm:chi_properties}
     \begin{itemize}
         \item If $SVC(x)$ and there is a surjection from $y$ onto $x$, then $SVC(y)$.
         \item Let $\alpha$ be an ordinal, if $SVC(x\times\alpha)$ then $SVC(x)$.
     \end{itemize}
\end{lemma}

\begin{proof}
    Suppose $SVC\left(x\right)$ and let $f:y\twoheadrightarrow x$ be a surjection. Let $z$ be a set and take a surjection $g:x\times\alpha\twoheadrightarrow z$ for some ordinal $\alpha$. Define a surjection from $y\times\alpha$ onto $z$ by:
    \begin{align*}
        &g':y\times \alpha\twoheadrightarrow z\\
        &g'(l,r):=g(f(l),r)
    \end{align*}
    Therefore $SVC(y)$ holds.
    \newline Let $\alpha\in\mathbf{On}$ (the class of Ordinals) and suppose $SVC\left(x\times\alpha\right)$. Let $y$ be a set and take a surjection $g:x\times\alpha\times\beta\twoheadrightarrow z$ for some ordinal $\beta$. $\left|\alpha\times\beta\right|=\left|\alpha\cdot \beta\right|$, $\alpha\cdot \beta$ being the ordinal product, so take a bijection $f:\alpha\cdot\beta\leftrightarrow \alpha\times\beta$. Define a surjection from $x\times\left(\alpha\cdot\beta\right)\twoheadrightarrow y$ by:
    \begin{align*}
    &g':x\times\left(\alpha\cdot\beta\right)\twoheadrightarrow y\\
    &g'(l,r):=g(l,f(r))
    \end{align*}
    Therefore $SVC(x)$ holds.
\end{proof}

We have the following theorem by Brunner. 
\begin{theorem}[Brunner \cite{brunner_fraenkel-mostowski_1990}]
    \label{thm:finite_support_first_order}
    If:
    \begin{itemize}
        \item $\mathcal{M}\models ZFA+$`Global Choice'; and
        \item $\mathcal{N}$ is a permutation model of $\mathcal{M}$;
    \end{itemize}
    then
    \begin{align*}
        &\mathcal{N}\text{ is a finite-support permutation model of }\mathcal{M}\\
        \text{iff } &\mathcal{N}\models SVC\left(\mathbb{A}^{<\omega}\right)
    \end{align*}
\end{theorem}
\begin{remark}
    `Global Choice' has multiple potential meanings, but is used here to say that there is a class surjection from $\mathbf{On}$ onto the universe. This is not a first order statement, as it is a disjunction over all the `Global Choice' axioms e.g. $V=L$.
\end{remark}
We can significantly weaken the hypotheses by removing intermediate steps from the proof structure. We begin by relating $SVC$ in a permutation model to the structure of the normal filter.
\begin{theorem}
    \label{thm:filterbase_by_chi}
    If:
    \begin{itemize}
        \item $\mathcal{M}\models ZFA$;
        \item $G\leqslant Sym(\mathbb{A})$;
        \item $\mathcal{F}$ is a normal filter on $G$ with $\mathcal{F}=\mathcal{F}_{nat}$ (Definition \ref{def:natural_filter});
        \item $\mathcal{N}=PM(\mathcal{M},G,\mathcal{F})$; and
        \item $S\in\mathcal{N}$;
    \end{itemize}
    then
    \begin{itemize}
        \item If $\mathcal{N}\models SVC(S)$, then there exists $H\in \mathcal{F}$ s.t. $\left\{H_x:x\in S\right\}$ is a filter base for $H\#\mathcal{F}$ (recall $\#$ from Example \ref{examples:filters}). 
        \item If $\mathcal{M}\models SVC(K)$ and there exists $H\in \mathcal{F}$ s.t. $\left\{H_x:x\in S\right\}$ is a filter base for $H\#\mathcal{F}$, then $\mathcal{N}\models SVC(S\times K)$.
    \end{itemize}
\end{theorem}
\begin{remark}
    The assumption that $\mathcal{F}=\mathcal{F}_{nat}$ is harmless as Remark \ref{rem:natural_filter} allows us to replace $\mathcal{F}$ by $\mathcal{F}_{nat}={\left(\mathcal{F}_{nat}\right)}_{nat}$.
\end{remark}
\begin{proof}
    We first find some set $B\in\mathcal{N}$, `witnessing' all the stabilisers (if $x\in\mathcal{N}$ then there is $y\in \mathcal{N}$ s.t. $G_x=G_y$). In Brunner's proof he uses $AC$ to pick elements, we avoid using $AC$ by taking some truncation of the universe via the cumulative hierarchy.
    \newline For $H\leqslant G$, s.t. $H=G_{x'}$ for some $x'\in \mathcal{N}$, define:
    \begin{equation*}
        \alpha_H:=min\left\{\alpha\in \mathbf{On}:\exists x\in V(\mathbb{A})_\alpha\cap\mathcal{N}. \space G_x=H\right\}
    \end{equation*}
    So there is $x\in V(\mathbb{A})_{\alpha_H}\cap\mathcal{N}$ with $H=G_x$.
    \newline Now define:
    \begin{equation*}
        \alpha=sup\left\{\alpha_H:H\leqslant G \text{ and }\exists x\in \mathcal{N}.G_x=H\right\}
    \end{equation*}
    So that $\alpha\geq \alpha_H$ for all $H\leqslant G$ s.t. $H=G_{x'}$ for some $x'\in \mathcal{N}$. So,  for $H\leqslant G$ s.t. $H=G_{x'}$ for some $x'\in \mathcal{N}$, there is $x\in V(\mathbb{A})_{\alpha}\cap\mathcal{N}$ with $H=G_x$. So let
    \begin{equation*}
        B:=V(\mathbb{A})_\alpha\cap\mathcal{N}
    \end{equation*}
    $B$ is equivariant and is a subset of $\mathcal{N}$, so $B\in\mathcal{N}$. Also, $\left\{G_x|x\in B\right\}$ is a filter base for $\mathcal{F}_{nat}=\mathcal{F}$.
    
    Suppose that $\mathcal{N}\models SVC(S)$. So $\mathcal{N}$ has an ordinal $\beta$ and a surjection $f:S\times \beta\twoheadrightarrow B$. Let $H:=G_f$. $H\in\mathcal{F}$ as $f\in\mathcal{N}$.
    \newline Let $x\in \mathcal{N}$. There is $b\in B$ s.t. $G_x=G_b$ and there are $s\in S$ and $\alpha\in \beta$ with $f(s,\alpha)=b$. If $\pi\in G_s\cap H$ then $\pi(f)=f$, $\pi(s)=s$ and $\pi(\alpha)=\alpha$ so that $\pi(b)=b$. It follows that:
    \begin{equation*}
        H_x=G_x\cap H=G_b\cap G_f\geqslant G_s\cap H=H_s
    \end{equation*}
    $\mathcal{F}$ is natural so $\left\{G_x:x\in\mathcal{N}\right\}$ is a base for $\mathcal{F}$, so $\left\{H_x:x\in\mathcal{N}\right\}$ is a base for $H\#\mathcal{F}$. For $x\in\mathcal{N}$, $H_x=H_s$ for some $s\in S$ so $\left\{H_x:x\in\mathcal{N}\right\}=\left\{H_s:s\in S\right\}$ is a base for $H\#\mathcal{F}$.
    
    Suppose that $\mathcal{M}\models SVC(K)$ and there exists $H\in \mathcal{F}$ s.t. $\left\{H_x:x\in S\right\}$ is a filter base for $H\#\mathcal{F}$.
    \newline Observe that for all $s\in S$ and $x\in\mathcal{N}$, if $H_s\leqslant H_x$ then the set $H\cdot \left\{\left\langle s,x\right\rangle\right\}$ is a surjection from $H \cdot\left\{s\right\}$ onto $H\cdot\left\{x\right\}$.
    Let $y\in\mathcal{N}$ and let:
    \begin{equation*}
        Y:=\left\{\left\langle s,x\right\rangle\in K\times y : H_s\leqslant H_x\right\}
    \end{equation*}
    Note that $y=\pi_2\left(Y\right)$ as $\left\{H_x:x\in S\right\}$ is a filter base for $H\#\mathcal{F}$ and $H_x\in H\#\mathcal{F}$ for $x\in y$. By $SVC(K)$, $\mathcal{N}$ has an ordinal $\gamma$ and a surjection from $K\times \gamma$ onto $Y$, which we will write as $\left\langle a,\alpha\right\rangle\mapsto\left\langle s_{a,\alpha},x_{a,\alpha}\right\rangle$.
    So $\left\langle a,\alpha\right\rangle\mapsto H\cdot\left\{x_{a,\alpha}\right\}$ is onto the $H$-orbits intersecting $y$.
    Define a surjection onto $H\cdot y$ by:
    \begin{equation*}
        f':=\bigcup\left\{H\cdot\left\{\left\langle \left\langle s_{a,\alpha},a,\alpha\right\rangle,x_{a,\alpha}\right\rangle\right\}:\alpha\in\gamma,a\in K\right\}
    \end{equation*}
    Observe that the domain of $f'$ is a subset of $S\times K\times\gamma$ and its range is a superset of $y$. Also note that $G_{f'}=H$ so that $f'\in\mathcal{N}$. We will expand $f'$'s domain and shrink its range to achieve a surjection from $S\times K\times\gamma$ onto $y$.
    \newline If $y=\emptyset$ then $f'$ is the desired surjection.
    \newline Otherwise, let $y_0\in y$ and define a surjection $f:S\times K\times\gamma\twoheadrightarrow y$ in $\mathcal{N}$ by:
    \begin{equation*}
        f(x,a,\alpha):= \left\{\begin{array}{ll}
             f'(x,a,\alpha) & f'(x,a,\alpha)\in y \\
             y_0 & \text{otherwise}
        \end{array}\right.
    \end{equation*}
    Now for all $y\in\mathcal{N}$, $\mathcal{N}$ contains an ordinal $\gamma$ and a surjection from $S\times K\times\gamma$ onto $y$. Therefore $\mathcal{N}$ satisfies $SVC(S\times K)$.
\end{proof}

We will now deduce that being a finite-support model can be expressed as logical property of permutation models.
Notice that the strengthened result says that $SVC(\mathbb{A^{<\omega}})$ is stable under taking finite-support permutation models.
\newline \label{sec:ext_prop}This is where the generalisation to ${<}\kappa$-support permutation models ends; we subtly rely on the fact that $\mathbb{A}^{<\omega}$ is absolute with respect to permutation models (all permutation models contain the same $\mathbb{A}^{<\omega}$ as the original universe, and all believe that it is $\mathbb{A}^{<\omega}$). For general $\kappa$, this is not the case e.g. a finite-support permutation model believes $SVC\left(\mathbb{A}^{<\aleph_1}\right)$ by Lemma \ref{thm:chi_properties} and Corollary \ref{thm:finite_support_fo_improved} but is not (necessarily) a ${<}\aleph_1$-support permutation model.
\begin{theorem}
    \label{thm:kappa_support_fo}
    If:
    \begin{itemize}
        \item $\mathcal{M}\models ZFA$;
        \item $\mathcal{N}$ is a permutation model of $\mathcal{M}$; and
        \item $\kappa$ is an infinite cardinal;
    \end{itemize}
    then
    \begin{itemize}
        \item If $\mathcal{M}\models SVC(\mathbb{A}^{<\kappa})$ and $\mathcal{N}$ is a ${<}\kappa$-support permutation model of $\mathcal{M}$, then $\mathcal{N}\models SVC(\mathbb{A}^{<\kappa})$ and $\mathcal{M}\models \mathbb{A}^{<\kappa}\in\mathcal{N}$.
        \item If $\mathcal{M}\models\mathbb{A}^{<\kappa}\in\mathcal{N}$ and $\mathcal{N}\models SVC(\mathbb{A}^{<\kappa})$, then $\mathcal{N}$ is a ${<}\kappa$-support permutation model of $\mathcal{M}$.
    \end{itemize}
\end{theorem}

\begin{proof}
    Suppose that $\mathcal{N}=PM(\mathcal{M},G,\mathcal{F})$ and assume that $\mathcal{F}=\mathcal{F}_{nat}$, as $\mathcal{F}_{nat}=\left(\mathcal{F}_{nat}\right)_{nat}$ and Remark \ref{rem:natural_filter}.
    
    Suppose that $\mathcal{M}\models SVC(\mathbb{A}^{<\kappa})$ and $\mathcal{F}=\mathcal{F}^G_{<\kappa}$, so $\left\{G_f:f\in\mathbb{A}^{<\kappa}\right\}$ is a filter base for $\mathcal{F}$. The elements of $\mathbb{A}^{<\kappa}$ are supported by themselves, so $\mathcal{M}\models\mathbb{A}^{<\kappa}\subseteq\mathcal{N}$ and $\mathbb{A}^{<\kappa}$ is equivariant so $\mathcal{M}\models\mathbb{A}^{<\kappa}\in\mathcal{N}$.
    \newline By Theorem \ref{thm:filterbase_by_chi}, $\mathcal{N}\models SVC(\mathbb{A}^{<\kappa}\times\mathbb{A}^{<\kappa})$. We now construct a surjection from $\mathbb{A}^{<\kappa}\times\kappa$ onto $\mathbb{A}^{<\kappa}\times\mathbb{A}^{<\kappa}$ in $\mathcal{N}$. Let
    \begin{align*}
    &f:\mathbb{A}^{<\kappa}\times\kappa\twoheadrightarrow\mathbb{A}^{<\kappa}\times\mathbb{A}^{<\kappa}\\
    &f(\mathbf{a},\alpha):=\left\langle \mathbf{a}_{i<min\left(\left|\mathbf{a}\right|,\alpha\right)},\mathbf{a}_{min\left(\left|\mathbf{a}\right|,\alpha\right)\leq i<\left|\mathbf{a}\right|}\right\rangle
    \end{align*}
    Given a sequence indexed by an ordinal ${<}\kappa$ and an ordinal $\alpha<\kappa$, $f$ splits the sequence into two sequences indexed by ordinals ${<}\kappa$ by taking the first $\alpha$ elements as the first sequence and the remaining elements as the second sequence. This is surjective, as if $\left(a_i\right)_{i\in\alpha}$ and $\left(b_i\right)_{i\in\beta}$ are two sequences in $\mathbb{A}$ with $\alpha,\beta<\kappa$, then $\alpha+\beta<\kappa$ so $f({\left(a_i\right)_{i\in\alpha}}^\frown\left(b_i\right)_{i\in\beta} )=\left\langle \left(a_i\right)_{i\in\alpha},\left(b_i\right)_{i\in\beta}\right\rangle$. Now by Lemma $\ref{thm:chi_properties}$, $\mathcal{N}\models SVC(\mathbb{A}^{<\kappa}\times\kappa)$ so $\mathcal{N}\models SVC(\mathbb{A}^{<\kappa})$.

    Suppose that $\mathcal{N}\models SVC(\mathbb{A}^{<\kappa})$ and $\mathcal{M}\models\mathbb{A}^{<\kappa}\in\mathcal{N}$. By Theorem \ref{thm:filterbase_by_chi}, there is $H\in\mathcal{F}$ s.t. $\left\{H_f | f\in\mathbb{A}^{<\kappa}\right\}$ is a filter base of $H\#\mathcal{F}$ so $H\#\mathcal{F}=\mathcal{F}^H_{<\kappa}$. Hence $\mathcal{N}=PM(\mathcal{M},H,\mathcal{F}^H_{<\kappa})$.
    Let $x$ be an element of the universe. If $H_x\in \mathcal{F}^H_{<\kappa})$ then, as $\mathcal{F}^H_{<\kappa}=H\#\mathcal{F}\subseteq\mathcal{F}$ and $H_x\leqslant G_x$, $G_x\in\mathcal{F}$. If $G_x\in\mathcal{F}$, then $H_x=H\cap G_x \in H\#\mathcal{F}= \mathcal{F}^H_{<\kappa}$. Hence $\mathcal{N}=PM(\mathcal{M},H,\mathcal{F}^H_{<\kappa})$.
\end{proof}
\begin{corollary}
    \label{thm:finite_support_fo_improved}
    If:
    \begin{itemize}
        \item $\mathcal{M}\models ZFA$; and
        \item $\mathcal{N}$ is a permutation model of $\mathcal{M}$;
    \end{itemize}
    then
    \begin{itemize}
        \item If $\mathcal{M}\models SVC(\mathbb{A}^{<\omega})$ and $\mathcal{N}$ is a finite-support permutation model of $\mathcal{M}$, then $\mathcal{N}\models SVC(\mathbb{A}^{<\omega})$.
        \item If $\mathcal{N}\models SVC(\mathbb{A}^{<\omega})$, then $\mathcal{N}$ is a finite-support permutation model of $\mathcal{M}$.
    \end{itemize}
\end{corollary}
\begin{proof}
    $\mathcal{M}\models \mathbb{A}^{<\omega}\in\mathcal{N}$ is always true, so apply Theorem \ref{thm:kappa_support_fo} for $\aleph_0$.
\end{proof}
\subsubsection{Orbit-Finite}

As our aim is to identify the logic of different finite-support permutation models so that we can transfer theorems about orbit-finite constructions, we want to be able to express orbit-finiteness. For a set, $x$, the property of being orbit-finite is defined dependent on the group inducing the permutation model, so we cannot directly refer to it inside of the finite-support permutation model. However, the following result says that, inside a finite-support permutation model, there is a property of sets that characterises when a set is orbit-finite with respect to the group inducing the permutation model.

\begin{definition}
\label{def:dedekind-finite}
    A set $x$ is Dedekind-finite if there exists an injection from $\mathbb{N}$ onto $x$.
\end{definition}
\begin{theorem}[Blass \cite{blass_power-dedekind_nodate}]
    \label{thm:orbit_finite_internal}
    If:
    \begin{itemize}
        \item $\mathcal{M}\models ZFA+AC$;
        \item $G\leqslant Sym(\mathbb{A})$ is oligomorphic;
        \item $\mathcal{N}=PM(\mathcal{M},G,\mathcal{F}^G_{fin})$; and
        \item $x\in\mathcal{N}$ is a set;
    \end{itemize}
    then
    \begin{align*}
        &\mathcal{N}\models\text{`$\mathcal{P}(x)$ is Dedekind-finite'}\\
        \text{iff }&x\text{ is orbit-finite with respect to }G
    \end{align*}
\end{theorem}
If we drop the assumption that $G$ is oligomorphic, then we must replace `$x$ is orbit-finite' by `for all $\mathbf{a}\in\mathbb{A}^{<\omega}$, $x$ is contained within a finite union of $G_\mathbf{a}$-orbits' (Blass uses the equivalent definition of orbit-finite from Lemma \ref{thm:orbit_finite_equiv}). We want oligomorphicity to expressible in a permutation model; we use Remark \ref{rem:orbit-finite2_implies_orbit-finite1}. We cannot use the above theorem directly, as we don't know a priori that $G$ is oligomorphic.

\begin{theorem}
    \label{thm:oligomorphic_pm}
    If:
    \begin{itemize}
        \item $\mathcal{M}\models ZFA+AC$;
        \item $G\leqslant Sym(\mathbb{A})$; and
        \item $\mathcal{N}=PM(\mathcal{M},G,\mathcal{F}^G_{fin})$;
    \end{itemize}
    then
    \begin{align*}
        &G\text{ is oligomorphic}\\
        \text{iff }&\text{for all } n\in\mathbb{N}\text{, }\mathcal{N}\models\text{`$\mathcal{P}(\mathbb{A}^n)$ is Dedekind-finite'}
    \end{align*}
\end{theorem}

\begin{proof}
    $\left(\implies\right)$: Suppose that $G$ is oligomorphic so for all $n\in\mathbb{N}$, $\mathbb{A}^n$ is a finite union of $G$-orbits so $\mathbb{A}^n$ is orbit-finite so $\mathcal{N}\models\text{\it{`$\mathcal{P}(\mathbb{A}^n)$ is Dedekind-finite'}}$ by Theorem \ref{thm:orbit_finite_internal}.
    \newline $\left(\impliedby\right)$: Suppose that for all $n\in\mathbb{N}$, $\mathcal{N}\models\text{\it{`$\mathcal{P}(\mathbb{A}^n)$ is Dedekind-finite'}}$. So for $n\in\mathbb{N}$, $\mathbb{A}^n$ is contained within a finite union of $G_\mathbf{a}$-orbits so, by Remark \ref{rem:orbit-finite2_implies_orbit-finite1}, $\mathbb{A}^n$ is a finite union of $G$-orbits. Hence, $G$ is oligomorphic.
\end{proof}

\subsection{Logical Truth}

Now that we can refer to being `orbit-finite' from inside a finite-support permutation model, we can define orbit-finite structures. See \cite{bojanczyk_slightly_nodate} and \cite{pitts_nominal_2013} for more details. To design these constructions, we want to refer to the structure on the atoms that induces the permutation model; this property is expressible, and we reserve the details for Section \ref{sec:Almost-homogeneous_forcing}.
\newline If we want to define constructions inside the finite-support permutation models induced by a structure $\mathcal{A}$, then it is crucial that model truth in the permutation model is the same as model truth in the original universe. This is because the satisfaction relation, $\models$, is $\Delta^{ZFA}_1$ (Definition \ref{def:absoluteness}).
\begin{remark}
    If:
    \begin{itemize}
        \item $\mathcal{M}\models ZFA$;
        \item $\mathcal{A}$ is a structure on $\mathbb{A}$;
        \item $\mathcal{N}=PM(\mathcal{M},Aut(\mathbb{A}),\mathcal{F}^{Aut(\mathbb{A})}_{fin})$;
        \item $\phi$ is a formula in the language of $\mathcal{A}$; and
        \item $\mathbf{a}\in\mathbb{A}^{<\omega}$;
    \end{itemize}
    then
    \begin{equation*}
        \mathcal{M}\models\left(\mathcal{A}\models\phi\left(\mathbf{a}\right)\right)\iff \mathcal{N}\models\left(\mathcal{A}\models\phi\left(\mathbf{a}\right)\right)
    \end{equation*}
    So truth in $\mathcal{A}$ is expressible in $\mathcal{N}$.
\end{remark}

\begin{remark}\label{rem:permutation_model_expressive}
    When $\mathcal{A}$ is oligomorphic, any definition or statement that is $\Delta^{ZFA}_1$, excepting reference to orbit-finiteness (formally, we extend the language by a unary predicate to refer to orbit-finiteness and define the $\Delta^{ZFA}_1$ formulae in this language), can be expressed in $\mathcal{N}$. This principle covers orbit-finite models of computations and their actions e.g. Turing machines and their languages.
\end{remark}

\subsection{Almost-homogeneous Forcing}
\label{sec:Almost-homogeneous_forcing}
We now state an important theorem.
\begin{theorem}[Eric J. Hall \cite{hall_generic_2000}]
    \label{thm:perm_forcing}
    If:
    \begin{itemize}
        \item $\mathcal{M}\models ZFA+AC$;
        \item $\mathcal{N}\subseteq\mathcal{M}$ is transitive;
        \item $\mathcal{N}\models ZFA$; and
        \item $\mathcal{N}$ and $\mathcal{M}$ have the same pure sets and the same set of atoms;
    \end{itemize}
    then
    \begin{align*}
        &\mathcal{N}\text{ is a permutation model of }\mathcal{M}\\
        \text{iff }&\mathcal{M}\text{ is a generic extension of }\mathcal{N}\text{ by some almost-homogeneous notion of forcing}
    \end{align*}
\end{theorem}

\subsubsection{Forcing}
We assume familiarity with standard forcing techniques. Detail and intuition can be found in \cite{jech_set_2013} and detail for $ZFA$ in \cite{blass_freyds_1989}.
\begin{definition}
\label{def:forcing_poset}
    $\left(\mathbb{P},\leq,\mathbf{1}\right)$ is a forcing poset if:
    \begin{itemize}
        \item $\leq$ is a preorder on $\mathbb{P}$;
        \item $\mathbf{1}\in\mathbb{P}$ is the largest element; and
        \item for all $p\in\mathbb{P}$, there exist $q,r\in\mathbb{P}$ with $q,r\leq p$ and for no $s\in\mathbb{P}$ is $s\leq q$ and $s\leq r$.
    \end{itemize}
    Elements of $\mathbb{P}$ are forcing conditions.
\end{definition}
\begin{definition}
    For $p,q\in\mathbb{P}$ we say:
\begin{align*}
    &p\parallel q\\
    \text{iff }&\text{there exists }r\in\mathbb{P}\text{ with }r\leq p,q
\end{align*}
i.e. `$p,q$ are compatible'.
\end{definition}
\begin{definition}
\label{def:p-names}
    Let $\left(\mathbb{P},\leq,\mathbf{1}\right)$ be a forcing poset. The class of $\mathbb{P}$-names, $\mathcal{V}^{(\mathbb{P})}$, is defined by transfinite recursion to contain all sets $A\subseteq \mathcal{V}^{(\mathbb{P})}\times \mathbb{P}$.
    $\mathbb{P}$-names are denoted by a dot e.g. $\dot x$.
    \newline We give a name to each element of the universe by recursion:
    \begin{equation*}
        \hat{x}:=\left\{\left\langle \hat y,\mathbf{1}\right\rangle:y\in x\right\}
    \end{equation*}
\end{definition}
\begin{definition}
\label{def:forcing_relation}
    We construct the forcing relation, $\Vdash$, from the forcing conditions to formulae in the language of forcing (this includes a $\in$ relation and allows reference to $\mathbb{P}$-names and a name for the ground model).
    \newline If $p$ is stronger than $q$ ($p\leq q$) and $q\Vdash \phi$, then $p\Vdash \phi$.
    \newline If $\mathcal{V}\models\phi$ then $\mathbf{1}\Vdash \phi^{\dot{\mathcal{V}}}$.
    \newline For $p\in\mathbb{P}$, the formulae forced by $p$ are deductively closed and consistent.
\end{definition}
\begin{theorem}
    If $\left(\mathbb{P},\leq,\mathbf{1}\right)$ is a forcing poset in $\mathcal{V}\models ZF(C)$, then:
    \begin{equation*}
        \mathbf{1}\Vdash ZF(C)
    \end{equation*}
\end{theorem}
\begin{definition}
\label{def:generic_object}
    $\Gamma$ is a generic object for $\left(\mathbb{P},\leq,\mathbf{1}\right)$ if:
    \begin{itemize}
        \item $\Gamma\subseteq \mathbb{P}$;
        \item $\mathbf{1}\in \Gamma$;
        \item for all $p\in \Gamma$ and $p\leq q$, then $q\in \Gamma$;
        \item for all $p,q\in \Gamma$, then there exists $r\in \Gamma$ with $r\leq p,q$; and
        \item if $D\in\mathcal{V}$ is dense in $\mathbb{P}$, then $\Gamma\cap D\not =\emptyset$.
        \begin{itemize}
            \item where $D$ is dense in $\mathbb{P}$ means that for all $p\in\mathbb{P}$ there exists $q\in D$ with $q\leq p$.
        \end{itemize}
    \end{itemize}
\end{definition}.
Let
\begin{equation*}
    \underline{\Gamma}:=\left\{\left\langle \hat p,p\right\rangle:p\in\mathbb{P}\right\}
\end{equation*}
then
\begin{equation*}
    \mathbf{1}\Vdash \underline{\Gamma}\text{ is a generic object for }\mathbb{P}
\end{equation*}
We can abuse notation and reason in the forcing extension $\mathcal{V}[\Gamma]$, which is the universe where we have added some generic object, $\Gamma$, to $\mathcal{V}$. $\mathcal{V}[\Gamma]$ will satisfy the formulae forced by $\mathbf{1}$.
We can now give interpretations for names recursively:
\begin{equation*}
    Val_\Gamma(u):=\left\{Val_\Gamma(v):\exists p\in G.\left\langle v,p\right\rangle\in u\right\}
\end{equation*}
For $x\in\mathcal{V}$, $Val_\Gamma(\hat x)=x$.
\newline Let us now consider forcing in models of $ZFA$.
\newline The final part in explaining Theorem \ref{thm:perm_forcing} is to define what makes a forcing poset `almost-homogeneous'.
\begin{definition}
\label{def:almost_homogeneous}
    If:
    \begin{itemize}
        \item $\mathcal{N}\models ZFA$; and 
        \item $\left(\mathbb{P},\leq,\mathbf{1}\right)$ is a forcing poset;
    \end{itemize}
    then
    \begin{align*}
        & \mathbb{P}\text{ is almost-homogeneous}\\
        \text{iff }&\text{for all }p,q\in\mathbb{P}\text{, there exists }\sigma\in Aut(\mathbb{P},\leq,\mathbf{1})\text{ s.t. }p\parallel \sigma q
    \end{align*}
\end{definition}
\subsubsection{Expressibility}

Forcing can be done in the logic of a model of $ZFA$. We can ask whether there is a notion of forcing that forces certain properties to hold, e.g. ask for the ground universe to be a permutation model of a certain shape. This means that statements about being a permutation model, induced in certain ways and with properties holding in the original universe, are expressible (assuming that $AC$ holds in the pure universe).
\newline Note that a model $\mathcal{N}\models ZFA$ could have been induced as a permutation model in multiple different universes.

%% file: text/sec5-cardinal.tex
\section{How to Have Less Things by Forgetting How to Count Them\protect\footnote{Title inspired by the paper `How to have more things by forgetting how to count them'\cite{karagila_how_2020}}}
\label{sec:cardinal}
\subsection{Summary}
In this section we:
\begin{itemize}
    \item Determine a list of conditions on the group and the atoms that can induce a forcing extension of the finite-support permutation model (Lemma \ref{thm:transfer_down}).
    \begin{itemize}
        \item This forcing extension will preserve the pure universe, assign a cardinality to the atoms, and have the finite-support permutation model as a finite-support permutation model.
    \end{itemize}
    \item Determine easy to check conditions guarantee Theorem \ref{thm:transfer_down}'s conditions (Theorem \ref{thm:transfer_down_sub} and Lemma \ref{lm:sufficient_group}).
    \item Prove that the first Fraenkel model is `unique' (Theorem \ref{thm:first_fraenkel_unique}).
    \begin{itemize}
        \item First Fraenkel models, on the same pure universe, are elementarily equivalent.
    \end{itemize}
    \item We prove that, under additional hypotheses, the structure inducing a finite-support permutation model stays the same after taking the forcing extension (Theorem \ref{thm:inducing_structure_after_transfer}).
    \item We determine conditions guaranteeing that a finite-support permutation model induced by some pure structure is equivalent to a finite-support permutation model induced by a different pure structure (Theorem \ref{thm:full_transfer}).
    \item We conclude with examples (Section \ref{sec:examples}).
    \begin{itemize}
        \item The finite-support permutation models induced by $\left\langle \mathbb{Q},\leq\right\rangle$ and $\left\langle \mathbb{R},\leq\right\rangle$ are elementarily equivalent, allowing reference to the inducing structures.
    \end{itemize}
\end{itemize}

\subsection{Plan}

We wish to identify the logic of different finite-support permutation models on the same pure universe. We proceed as follows:
\begin{enumerate}
    \item Use forcing to make one of the finite-support permutation models to `imagine' that it was induced as a different finite-support permutation model.
    \item By Remark \ref{rem:pm_well-defined}, this construction is unique up-to class isomorphism, so the permutation models are elementarily equivalent (allowing the inducing structure as a parameter).
\end{enumerate}

\subsection{Transfer Downwards}

There is the basic example, when using forcing, of adding a bijection between two sets of different cardinality. This is called `Cardinal Collapse'.
\newline In a finite-support permutation model, we forget the cardinality of the atoms hence, we can to collapse them to a new cardinality while maintaining `enough' structure of the universe. What this means is the following, let:
\begin{itemize}
    \item $\mathcal{M}\models ZFA+{AC}^{pure}$; and
    \item $\mathcal{N}$ be a finite-support permutation model of $\mathcal{M}$.
\end{itemize}
We wish to find a notion of forcing $\left(\mathbb{P},\leq,\mathbf{1}\right)\in\mathcal{N}$ s.t.:
\begin{itemize}
    \item $\left(\mathbb{P},\leq,\mathbf{1}\right)\in\mathcal{N}$ is almost-homogeneous;
    \item forcing preserves the pure universe (adds no pure sets); and
    \item $\mathbf{1}\Vdash \left|\hat{\mathbb{A}}\right|=\hat{\kappa}$ for some cardinal $\kappa$.
\end{itemize}

Lemma \ref{thm:transfer_down} gives us sufficient conditions for the above to hold. We will follow Lemma \ref{thm:transfer_down} with theorems that provide easier conditions to verify.

\begin{lemma}
    \label{thm:transfer_down}
    If:
    \begin{enumerate}
        \item $\left\langle\mathcal{M},\mathbb{A},\in\right\rangle\models ZFA+{AC}^{pure}$;
        \item $G\leqslant Sym(\mathbb{A})$;
        \item $\mathcal{N}=PM(\mathcal{M},G,\mathcal{F}_{fin})$;
        \item $\kappa$ is a cardinal in $\mathcal{M}$;
        \item $t_\alpha\subseteq \mathbb{A}^\alpha$ for $\alpha\in\kappa$ is s.t.:
        \begin{enumerate}
            \item \label{hyp:transfer_down_orbits}for all $\alpha\in\kappa$, $t_\alpha$ is a $G$-orbit;
            \item \label{hyp:transfer_down_injective}for all $f\in \bigcup_{\alpha\in\kappa}t_\alpha$, $f$ is injective; and
            \item \label{hyp:transfer_down_extensions}for all $\alpha\leq\beta\in\kappa$, $t_\alpha=\left\{f\restriction\alpha:f\in t_\beta\right\}$;
        \end{enumerate}
        \item $\mathbb{P}:=\bigcup_{\alpha\in\kappa,f\in t_\alpha} \mathcal{P}_{fin}(f)$ is s.t.:
        \begin{enumerate}
            \item \label{hyp:transfer_down_almost-hom}for all $ f,g\in\mathbb{P}$, there exists $\sigma\in Aut(\left(\mathbb{P},\supseteq\right))\cap \mathcal{N}$ and $h\in\mathbb{P}$ with $\sigma f,g\subseteq h$; and
            \item \label{hyp:transfer_down_extend_condition}for all $f\in\mathbb{P},a\in\mathbb{A}$, there exists $ g\in\mathbb{P}$ with $f\subseteq g$ and $a$ is in the range of $g$;
        \end{enumerate}
    \end{enumerate}
    then forcing in $\mathcal{N}$ with $\left(\mathbb{P},\supseteq,\emptyset\right)$ forces $ZFA+AC+\left|\mathbb{A}\right|=\kappa$ and $\mathcal{N}$ is a permutation model in the forcing extension.
    \newline If $\mathcal{M}\models SVC(\mathbb{A}^{<\omega})$ (Definition \ref{def:SVC}), then this is a finite-support permutation model.
\end{lemma}

\begin{proof}
    \begin{claim}
        $\mathbb{P}$ is equivariant and $\mathbb{P\in\mathcal{N}}$
    \end{claim}
    \begin{proof}
        $\mathbb{P}$ is equivariant by construction and for $x\in\mathbb{P}$, $x$ is hereditarily finite so $x\in\mathcal{N}$. Hence, $\mathbb{P}\in\mathcal{N}$.
    \end{proof}
    \begin{claim}
        Forcing with $\left(\mathbb{P},\supseteq,\emptyset\right)$ adds no pure sets.
    \end{claim}
    \begin{proof}
        Let:
        \begin{itemize}
            \item $\dot{S}$ be pure name;
            \item $p\in\mathbb{P}$; and
            \item $X$ be a pure set.
        \end{itemize}
        Suppose that
        \begin{equation*}
            p\Vdash \dot S \subseteq \hat X\text{ and }p\Vdash \dot S\textit{ is pure}
        \end{equation*}
        By \eqref{hyp:transfer_down_extend_condition}, we can strengthen $p$ to $p'$, supporting $\dot S$.
        \newline Let $x\in X$.
        \newline Suppose that there exists $q_0,q_1\in \mathbb{P}$ with:
        \begin{itemize}
            \item $p'\subseteq q_0,q_1$;
            \item $q_o \Vdash \hat x \in \dot S$; and
            \item $q_1 \Vdash \hat x \notin \dot S$.
        \end{itemize}
        Let:
        \begin{equation*}
            \alpha:=max(dom(q_0),dom(q_1))
        \end{equation*}
        There exists $q_0',q_1'\in t_\alpha$ with $q_0\subseteq q_0'$ and $q_1 \subseteq q_1'$. There exists $\sigma\in G$ with $\sigma q_0'=q_1'$. Thus, $\sigma q_0,q_1\subseteq q_1'$ and $\sigma p'=p'$, as $p'$ is common to $q_0$ and $q_1'$. Now, $\sigma q_0 \Vdash \hat x \in \dot S$, as $\Vdash$ is equivariant and $q_1 \Vdash \hat x \notin \dot S$. This is a contradiction as they have a common strengthening.
        \newline Let:
        \begin{equation*}
            A:=\left\{x\in X:p'\Vdash \hat x\in \dot S \right\}
        \end{equation*}Every extension of $p'$ agrees on the elements of $\dot S$ so:
        \begin{equation*}
            p'\Vdash \dot {S}=\hat A\in \hat{\mathcal{N}}
        \end{equation*}
        Every forcing condition, $p'$, strengthening $p$ to support $\dot S$ has $p'\Vdash \dot{S} \in \hat{\mathcal{N}}$ so $p\Vdash \dot S\in\hat{\mathcal{N}}$. So
        \begin{equation*}
            \emptyset\Vdash \textit{if }X\in \hat{\mathcal{N}}\textit{ and }\dot S\subseteq X\textit{, then }\dot S \in \hat{\mathcal{N}}
        \end{equation*}
        \newline So the generic extension of $\mathcal{N}$ given by $\left(\mathbb{P},\supseteq, \emptyset\right)$ has the same pure sets as $\mathcal{N}$.
        \newline Formally, we use well-foundedness to pick $\dot S$ of minimum rank that can be forced to be a new pure set. $\dot S$ must appear as a subset of a pure set of lower rank, which must already exist, leading to a contradiction. However, this is omitted from the proof as this is standard (See \cite{hall_generic_2000} for examples).
    \end{proof}
    \begin{claim}
        Forcing with $\left(\mathbb{P},\supseteq,\emptyset\right)$ adds a bijection $f:\kappa\rightarrow\mathbb{A}$.
    \end{claim}
    \begin{proof}
        Let:
        \begin{equation}\label{eq:forcing_bijection}
            \dot f := \left\{\left\langle\hat q,p\right\rangle|p\in\mathbb{P}\text{ and }p=\left\{q\right\}\right\}
        \end{equation}
        We show that $\emptyset\Vdash\dot f \text{ is a bijection }\hat \kappa \rightarrow \hat {\mathbb{A}}$.
        \newline Let $\Gamma\subseteq\mathbb{P}$ be generic and we work in $\mathcal{N}[\Gamma]$. Let:
        \begin{equation*}
            f:=Val_{\Gamma}\left(\dot f\right)
        \end{equation*}
        By \eqref{eq:forcing_bijection},
        \begin{equation*}
            f\subseteq\kappa\times\mathbb{A}
        \end{equation*}
        Let $\alpha\in \kappa$ and let:
        \begin{equation*}
            D_\alpha:=\left\{p\in\mathbb{P}:\alpha\in dom(p)\right\}
        \end{equation*}
        We show that $D_\alpha$ is dense (Definition \ref{def:generic_object}). Let $h\in\mathbb{P}$ so there exists $\beta\in\kappa$ and $h'\in t_\beta$ with $h\subseteq h'$. Hence, there exists $h''\in t_{\left(\alpha+1\right)\cup\beta}$ with $h'\subseteq h''$ by \eqref{hyp:transfer_down_extensions}. Now, $h\subseteq h''\restriction\left({dom(h)\cup\left\{\alpha\right\}}\right)$ and $h''\restriction\left({dom(h)\cup\left\{\alpha\right\}}\right)\in D_\alpha$. So $D_\alpha$ is dense.
        \newline So there is $p\in D_\alpha\cap\Gamma$, so $\left\langle \alpha,p\left(\alpha\right)\right\rangle\in f$.
        This holds for all $\alpha\in\kappa$ so:
        \begin{equation*}
            dom(f)=\kappa
        \end{equation*}
        Let $\alpha\in\kappa$ and let:
        \begin{equation*}
            D_\alpha:=\left\{p\in \mathbb{P}:a\in ran(p)\right\}
        \end{equation*}
        $D$ is dense as for $h\in \mathbb{P}$, there is $h'\in\mathbb{P}$ s.t. $h\subseteq h'\in D_\alpha$, by \eqref{hyp:transfer_down_extend_condition}.
        So there exists $p\in D_\alpha \cap \Gamma$, so $\left\langle p^{-1}\left(a\right),a\right\rangle\in f$. This holds for all $a\in A$ so:
        \begin{equation*}
            ran(f)=\mathbb{A}
        \end{equation*}
        If $p\in\mathbb{P}$ strengthens both $\left\{\left\langle x,a\right\rangle\right\}\in\Gamma$ and $\left\{\left\langle y,b\right\rangle\right\}\in\Gamma$ then, as $p$ is an injective function \eqref{hyp:transfer_down_injective}, $\left(x=y\text{ and }a=b\right)$ or $\left(x\neq y\text{ and }a\neq b\right)$.
        All elements of $\Gamma$ are comparable, so $f$ is an injective function.
        \newline Putting all of these together, $f:\kappa\leftrightarrow\mathbb{A}$ is a bijection.
    \end{proof}
    Now for all $\Gamma\subseteq\mathbb{P}$ generic:
    \begin{itemize}
        \item $\mathcal{N}[\Gamma]$ contains the same pure sets as $\mathcal{N}$;
        \item $\mathcal{N}[\Gamma]$ is an almost-homogeneous extension of $\mathcal{N}$; and
        \item $\mathcal{N}[\Gamma]$ contains a bijection $\kappa\rightarrow\mathbb{A}$, so satisfies AC.
    \end{itemize}
    So, by Theorem \ref{thm:perm_forcing}, $\mathcal{N}$ is a permutation model of $\mathcal{N}[\Gamma]$.
    \newline If $\mathcal{M}\models SVC(\mathbb{A}^{<\omega})$ then $\mathcal{N}\models SVC(\mathbb{A}^{<\omega})$, by Theorem \ref{thm:finite_support_fo_improved}, so $\mathcal{N}$ is a finite-support permutation model of $\mathcal{N}[\Gamma]$, by Theorem \ref{thm:finite_support_fo_improved}.
\end{proof}

\subsubsection{Simplifying Conditions}

Lemma \ref{thm:transfer_down} has many conditions. The most concerning is that $\mathbb{P}$ is almost-homogeneous \eqref{hyp:transfer_down_almost-hom}, which is the only property internal to $\mathcal{N}$.
\newline Ideally we want simpler conditions, and for \eqref{hyp:transfer_down_almost-hom} to be easier to check.
\newline We have two ways to satisfy condition \eqref{hyp:transfer_down_almost-hom}:
\begin{enumerate}
    \item Finding permutations on $\mathbb{A}$ as in Lemma \ref{lm:sufficient_group}, that remain in $\mathcal{N}$.
    \item Finding permutations on $\kappa$ as in Corollary \ref{thm:transfer_down_sub} and inspired by the proof of Theorem \ref{thm:perm_forcing}.
\end{enumerate}

One way to ensure the conditions hold, is to have $\mathbb{A}$ be `nicely' covered by some set of size $\kappa$ with respect to the group.

\begin{theorem}
    \label{thm:transfer_down_sub}
    If
    \begin{enumerate}
        \item $\left\langle\mathcal{M},\mathbb{A},\in\right\rangle\models ZFA+{AC}^{pure}$;
        \item $G\leqslant Sym(\mathbb{A})$;
        \item $\mathcal{N}=PM(\mathcal{M},G,\mathcal{F}_{fin})$;
        \item $\kappa$ is a cardinal;
        \item $\mathbb{B}\subseteq\mathbb{A}$ with:
        \begin{enumerate}
            \item \label{hyp:std_1} $\left|\mathbb{B}\right| = \kappa$;
            \item \label{hyp:std_2} for all $m\in\mathbb{N}\text{, }\mathbf{a},\mathbf{b}\in\mathbb{B}^m$, and $\sigma'\in G$, if $\sigma'\mathbf{a}=\mathbf{b}$ then there exists $\sigma\in G_\mathbb{B}$ with $\sigma\mathbf{a}=\mathbf{b}$; and
            \item \label{hyp:std_3}for all $\mathcal{S}\subseteq_{fin}\mathbb{B}$ and $\mathcal{S'}\subseteq_{fin}\mathbb{A}$ there exists $\sigma\in G$ s.t. for all $b\in\mathcal{S}$, $\sigma b=b$ $\sigma\left(\mathcal{S}'\right)\subseteq\mathbb{B}$;
        \end{enumerate}
    \end{enumerate}
    then $\mathcal{N}$ has a forcing extension satisfying $ZFA+AC+\left|\mathbb{A}\right|=\kappa$ and $\mathcal{N}$ is a permutation model in this forcing extension.
    \newline If $\mathcal{M}\models SVC(\mathbb{A}^{<\omega})$, then this is a finite-support permutation model.
\end{theorem}

\begin{proof}
    Let $j:\kappa\leftrightarrow\mathbb{B}$ be a bijection.
    For $\beta\in\kappa$, let:
    \begin{equation}\label{eq:def_of_type_forcing}
        t_\beta:=G\cdot (j\restriction{\beta})
    \end{equation}
    Now \eqref{hyp:transfer_down_orbits}, \eqref{hyp:transfer_down_injective}, and \eqref{hyp:transfer_down_extensions} of Lemma \ref{thm:transfer_down} hold.
    \newline Let $a\in\mathbb{A}$, $f\in\mathbb{P}$, and $S:= dom( f)$. By \eqref{eq:def_of_type_forcing}, there exists $\pi\in G$ with $f=\pi(j\restriction S)$.
    By \eqref{hyp:std_3}, there exists $\sigma\in G$ s.t. for $x\in j(S)$, $\sigma x=x$ and $\sigma\pi^{-1}a\in \mathbb{B}$. Define:
    \begin{equation}\label{eq:extend_f_forcing}
        f':=\pi(\sigma^{-1}(j\restriction\left({S\cup\left\{\sigma(j^{-1}(\pi^{-1}(a)))\right\}}\right)))
    \end{equation}
    So that $f\subseteq f'\in\mathbb{P}$ and $a\in ran(f')$.
    \newline So \eqref{hyp:transfer_down_extend_condition} of Lemma \ref{thm:transfer_down} holds.
    \newline
    For $\pi\in G_\mathbb{B}$, define:
    \begin{equation*}\label{eq:conjugate_group_element}
        \pi_\kappa=j^{-1}\circ\pi\circ j
    \end{equation*}
    Define
    \begin{equation*}\label{eq:conjugate_group}
        G_{\kappa}=\left\{\pi_{\kappa} :\pi \in G\right\}\in\mathcal{V}\subseteq\mathcal{N}
    \end{equation*}
    We can define an action of $G_\kappa$ on $\mathbb{P}$ by:
    \begin{equation*}\label{eq:conjugate_group_action}
        \pi p:=p\circ\pi^{-1}
    \end{equation*}
    To show that this is a group action, it suffices to show that for all $\pi\in G_\kappa$ and $p\in\mathbb{P}$, $\pi p\in\mathbb{P}$.
    \newline Let $\pi\in G_\mathbb{B}$ and $p\in\mathbb{P}$. There exists $\sigma\in G$ and $S\subseteq_{fin}\kappa$ with $p=\sigma\circ j\restriction S$ by \eqref{eq:def_of_type_forcing}. We calculate:
    \begin{align*}
        \pi_\kappa p&=\sigma\circ \left(j\restriction S\right)\circ \pi_\kappa^{-1}\\
        &=\sigma\circ\left(j\restriction S\right)\circ j^{-1}\circ\pi^{-1}\circ j\\
        &=\left(\sigma\circ\pi^{-1}\right)\circ\left(j\restriction\left(\left(j^{-1}\circ \pi \circ j\right)[S]\right)\right)\\
        &\in\mathbb{P}
    \end{align*}
    \begin{claim}
        Under this group action, $G_\kappa$ consists of automorphisms of $\left(\mathbb{P},\supseteq,\emptyset\right)$.
    \end{claim}
    \begin{proof}
        Let $f,g\in\mathbb{P}$ and $\pi\in G_\kappa$. If $f\subseteq g$ then $f\circ \pi^{-1}\subseteq g\circ\pi^{-1}$ so $\pi_\kappa f\subseteq\pi_\kappa g$.
        \newline Also, $\pi_\kappa\emptyset=\emptyset\circ \pi^{-1}=\emptyset$.
    \end{proof}
    \begin{claim}
        $\left(\mathbb{P},\supseteq,\emptyset\right)$ is almost-homogeneous.
    \end{claim}
    \begin{proof}
        Let $f,g\in\mathbb{P}$. There exists $\pi\in G$ with $\pi(ran(f)\cup ran(g))\subseteq\mathbb{B}$ by \eqref{hyp:std_3}. There exists $\sigma\in G$ and $S\subseteq_{fin}\kappa$ with $f=\sigma\circ j\restriction S$, by \eqref{eq:def_of_type_forcing}. Now $\left(\pi\circ \sigma\right)(j[S])\subseteq\mathbb{B}$, so there exists $\pi'\in G_\mathbb{B}$ with $\pi'\restriction j[S]=\pi\circ\sigma\restriction j[S]$ by \eqref{hyp:std_2}. We calculate:
        \begin{align*}
            \pi'_{\kappa} ( f) &= \sigma \circ j \circ j^{-1} \circ {\pi '}^{-1} \circ j \restriction {\left(j^{-1}\circ \pi' \circ j\right)[S]}\\
            &=\pi \circ j \restriction{S'}
        \end{align*}
        where $S'=\left(j^{-1}\circ \pi' \circ j\right)[S]\subseteq_{fin}\kappa$.
        \newline Similarly, find $\pi''_{\kappa}\in G_{\kappa}$ and $S''\subseteq_{fin}\kappa$ s.t.:
        \begin{equation*}
            \pi''_{\kappa}(g)=\pi\circ j \restriction {S''}
        \end{equation*}
        So:
        \begin{equation*}
            \pi'_{\kappa}f,\pi''_\kappa g\subseteq\pi\circ j \restriction \left({S'\cup S''}\right)\in \mathbb{P}
        \end{equation*}
        So $\left(\mathbb{P},\supseteq,\emptyset\right)$ is almost-homogeneous.
    \end{proof}
    So the conclusion of Lemma \ref{thm:transfer_down} holds.
\end{proof}

Theorem \ref{thm:transfer_down_sub} still requires us to identify a subset of $\mathbb{A}$ of known cardinality. This may be impossible e.g. if $\mathcal{N}$ is a finite-support permutation model in $\mathcal{N}'$ and $\mathcal{N}'$ is a permutation model of $\mathcal{M}$ then $\mathcal{N}'$ may believe that no subset of $\mathbb{A}$ in bijection with a cardinal. So Theorem \ref{thm:transfer_down_sub} is insufficient for identifying all first Fraenkel models.
\newline However, first Fraenkel models contain the finite permutations, which are dense in the topology of pointwise convergence. This is sufficient for $\mathbb{P}$ to be almost-homogeneous.
\begin{lemma}
    \label{lm:sufficient_group}
    Suppose that the hypotheses of Lemma \ref{thm:transfer_down} are met except \eqref{hyp:transfer_down_almost-hom}. If $\overline{G\cap\mathcal{N}}\supseteq G$ then \eqref{hyp:transfer_down_almost-hom} holds, where the closure is in the topology of pointwise convergence in $\mathcal{M}$.
\end{lemma}

\begin{proof}
    Let $p_0,p_1\in\mathbb{P}$ so there exist $\alpha\in\kappa$, $S_0\subseteq_{fin}\alpha$, $f_0\in t_\alpha$, $\beta\in\kappa$, $S_1\subseteq_{fin}\beta$, and $f_1\in t_\beta$ with $p_0=f_0\restriction S_0$ and $p_1=f_1\restriction S_1$. By \eqref{hyp:transfer_down_extensions}, assume that $\alpha=\beta$. There exists $\sigma\in G$ with $\sigma f_0=f_1$ by \eqref{hyp:transfer_down_orbits}. As $\overline{G\cap\mathcal{N}}\supseteq G$, there exists $\sigma'\in G\cap\mathcal{N}$ with $\sigma'\restriction\left(S_0\cup S_1\right)=\sigma\restriction\left(S_0\cup S_1\right)$. As $\pi'\in G\leqslant Aut(\left(\mathbb{P,\supseteq,\emptyset}\right))$, $\sigma'p_0,p_1\subseteq f_1\restriction\left({S_0\cup S_1}\right)\in\mathbb{P}$.
    \newline So, $\left(\mathbb{P},\supseteq,\emptyset\right)$ is almost-homogeneous.
\end{proof}

\subsubsection{The First Fraenkel Model is Unique}

We can now prove that:
\begin{theorem}
    \label{thm:first_fraenkel_unique}
    For all sentences $\phi$,
    \begin{equation*}
        ZFA+{AC}^{pure}\vdash\phi^{PM(\mathcal{V}(\mathbb{A}),Sym(\mathbb{A}),\mathcal{F}^{Sym(\mathbb{A})}_{fin})}\leftrightarrow\phi^{PM(\left\langle \mathbb{N}\right\rangle)}
    \end{equation*}
    i.e. if $ZFA+{AC}^{pure}$ hold then,
    \begin{equation*}
        PM(\mathcal{V}(\mathbb{A}),Sym(\mathbb{A}),\mathcal{F}^{Sym(\mathbb{A})}_{fin})\equiv PM(\left\langle \mathbb{N}\right\rangle)
    \end{equation*}
    where $PM(\left\langle \mathbb{N}\right\rangle)$ is the class from Definition \ref{def:permutation_model_from_structure}. The definition of $PM(\left\langle \mathbb{N}\right\rangle)$ is absolute between the universe and the pure universe, so all first Fraenkel models, with the same pure universe, are elementarily equivalent.
\end{theorem}
\begin{proof}
    Suppose that $\mathcal{M}\models ZFA+AC^{pure}$. Let:
    \begin{align*}
        \mathcal{V}&=\mathcal{V}(\mathbb{\emptyset})\\
        \mathcal{N}&=PM(\mathcal{M},Sym(\mathbb{A}),\mathcal{F}^{Sym(\mathbb{A})}_{fin})
    \end{align*}
    Note that every element of $\mathcal{N}$ is symmetric with respect to every group and this property is $\Pi_1$ so downwards absolute (Definition \ref{def:absoluteness}). So, the only permutation model of $\mathcal{N}$ is itself.
    \begin{claim}
        \label{claim:first_fraenkel_model_satisfies_transfer_down}
        $\mathcal{N}$ satisfies the conditions of Lemma \ref{thm:transfer_down}.
    \end{claim}
    Assuming Claim \ref{claim:first_fraenkel_model_satisfies_transfer_down}, apply Lemma \ref{thm:transfer_down} on $\mathcal{N}$ to obtain a forcing extension $\mathcal{N}[\Gamma]$, where $\mathcal{N}$ is a permutation model.
    \newline Let:
    \begin{equation*}
        \mathcal{N}'=PM(\mathcal{N}[\Gamma],Sym^{\mathcal{N}[\Gamma]}(\mathbb{A}),(\mathcal{F}^{Sym(\mathbb{A})}_{fin})^{\mathcal{N}[\Gamma]})
    \end{equation*}
    $\mathcal{N}'$ is the first Fraenkel model of $\mathcal{N}[\Gamma]\models\left|\mathbb{A}\right|=\aleph_0$, so $\mathcal{N}'\cong PM(\left\langle \mathbb{N}\right\rangle)$.
    All permutation models of $\mathcal{N}[\Gamma]$ contain $\mathcal{N}'$ as a permutation model, so $\mathcal{N}=\mathcal{N}'\cong PM(\left\langle \mathbb{N}\right\rangle)$.
    \newline We finish by noting that $PM\left(\left\langle \mathbb{N}\right\rangle\right)$ is a pure class and is absolute between $\mathcal{M}$ and $\mathcal{V}$.
\end{proof}
\begin{proof}[Proof of Claim \ref{claim:first_fraenkel_model_satisfies_transfer_down}]
    For $n\in\aleph_0$, let:
    \begin{equation*}
        t_n=\left\{f\in\mathbb{A}^n:f\text{ is injective}\right\}
    \end{equation*}
    By considering finite permutations on $\mathbb{A}$, \eqref{hyp:transfer_down_orbits},\eqref{hyp:transfer_down_injective}, and \eqref{hyp:transfer_down_extensions} clearly hold.
    \newline Let $f\in\mathbb{P}$ and $a\in\mathbb{A}$, so there exists $n\in\mathbb{N}$ and $f'\in t_n$ with $f=f'\restriction dom(f)$. If $a\in ran(f')$ then $f\subseteq f'\in \mathbb{P}$, otherwise $f\subseteq f'\cup\left\{\left\langle n,a\right\rangle\right\}\in\mathbb{P}$. So \eqref{hyp:transfer_down_extend_condition} holds.
    \newline Let
    \begin{equation*}
        Perm(\mathbb{A}):=\left\{\pi\in Sym(\mathbb{A}):\pi\text{ fixes cofinite elements}\right\}
    \end{equation*}
    The elements of $Perm(\mathbb{A})$ are hereditarily finitely supported by the finite set of atoms that they don't fix, so $Perm(\mathbb{A})\subseteq\mathcal{N}$.
    In the topology of pointwise convergence on $Sym(\mathbb{A})$,
    \begin{equation*}
        \overline{Perm(\mathbb{A})}=Sym(\mathbb{A})
    \end{equation*}
    So, by Lemma \ref{lm:sufficient_group}, \eqref{hyp:transfer_down_almost-hom} holds.
\end{proof}

\subsection{Group after Transfer}
Theorem \ref{thm:first_fraenkel_unique} is specific to the first Fraenkel model. However, we want to identify more interesting finite-support permutation models.
\newline After applying Lemma $\ref{thm:transfer_down}$, we want to know how the permutation model is induced in the new universe.
If:
\begin{itemize}
    \item $\mathcal{M}\models ZFA+AC^{pure}+SVC(\mathbb{A}^{<\omega})$;
    \item $\mathcal{N}$ is a finite-support permutation model of $\mathcal{M}$;
    \item $\mathcal{N}$ is a permutation model of the forcing extension $\mathcal{N}[\Gamma]\models \left|\mathbb{A}\right|=\kappa$; and
    \item $\mathcal{V}$ is the pure universe of $\mathcal{N}$;
\end{itemize}
then
\begin{itemize}
    \item $\mathcal{N}$ is a finite-support permutation model of $\mathcal{N}[\Gamma]$ by Theorem \ref{thm:finite_support_fo_improved};
    \item $\mathcal{N}$ is induced by a structure, $\mathcal{A}\in\mathcal{N}$ on $\mathbb{A}$, in $\mathcal{M}$ by Theorem \ref{thm:adequacy_automorphism};
    \item $\mathcal{A}$ can be assumed to have pure language as $\mathcal{N}[\Gamma]\models AC$; and
    \item $\mathcal{N}[\Gamma]\models\mathcal{A}\cong\mathcal{A}^*$ for some $\mathcal{A}^*\in\mathcal{V}$ as $\mathcal{N}[\Gamma]\models\left|\mathbb{A}\right|=\kappa$;
\end{itemize}
thus, $\mathcal{N}[\Gamma]\models \mathcal{N}\cong PM(\mathcal{A}^*)$.
\newline So, Lemma $\ref{thm:transfer_down}$ gives us conditions guaranteeing that a finite-support permutation model is elementarily equivalent to one induced by a structure of cardinality $\kappa$ in the pure universe.
\begin{remark}
    By Remark \ref{rem:permutation_model_expressive} and Theorem \ref{thm:oligomorphic_pm}, this allows us to transfer results about orbit-finite constructions.
\end{remark}
\begin{corollary}
    If:
    \begin{itemize}
        \item $\mathcal{M}\models ZFA+AC^{pure}+SVC(\mathbb{A}^{<\omega})$;
        \item $\mathcal{N}$ is a finite-support permutation model of $\mathcal{M}$; and
        \item Lemma \ref{thm:transfer_down} holds with $\kappa$ on $\mathcal{N}$;
    \end{itemize}
    then there exists a structure $\mathcal{A}^*\in \mathcal{V}$ s.t. $\left|\mathcal{A}^*\right|=\kappa$ and $\mathcal{N}\equiv PM(\mathcal{A}^*)$.
\end{corollary}
We do not know what $\mathcal{A}^*$ is. We want to say that $\left\langle PM(\left\langle\mathbb{Q},\leq\right\rangle),\underline{\left\langle\mathbb{Q},\leq\right\rangle}\right\rangle\equiv\left\langle PM(\left\langle\mathbb{R},\leq\right\rangle),\underline{\left\langle\mathbb{R},\leq\right\rangle}\right\rangle$, so we need to determine $\mathcal{A}^*$.
\begin{theorem}
    \label{thm:inducing_structure_after_transfer}
    If:
    \begin{itemize}
        \item $\left\langle\mathcal{M},\mathbb{A},\in\right\rangle\models ZFA$;
        \item $\left\langle\mathcal{M}',\mathbb{A},\in\right\rangle\models ZFA+AC$;
        \item $\mathcal{A}\in\mathcal{M}$ is a structure on $\mathbb{A}$;
        \item $\mathcal{N}:=PM(\mathcal{M},Aut(\mathcal{A}),\mathcal{F}^{Aut(\mathcal{A})}_{fin})^{\mathcal{M}}$;
        \item $\mathcal{A}$ has pure language;
        \item $\mathcal{M}\models$`$\mathcal{A}$ is $\aleph_0$-homogeneous'; and 
        \item $\mathcal{N}$ is a finite-support permutation model of $\mathcal{M}'$;
    \end{itemize}
    then
    \begin{equation*}
        \mathcal{N}=PM(\mathcal{M}',Aut(\mathcal{A}),\mathcal{F}^{Aut(\mathcal{A})}_{fin})^{\mathcal{M}'}
    \end{equation*}
\end{theorem}

\begin{proof}
    By Corollary \ref{thm:adequacy_automorphism}, there is $\mathcal{A}'\in\mathcal{N}$ s.t. 
    \begin{equation*}
        \mathcal{N}=PM(\mathcal{M}',Aut(\mathcal{A}'),\mathcal{F}^{Aut(\mathcal{A}')}_{fin})^{\mathcal{M}'}
    \end{equation*}
    $\mathcal{M}'\models AC$ so we may assume that the language of $\mathcal{A}'$ is pure, as if $\mathcal{A}''$ is $\mathcal{A}'$ with a relabelled language then $Aut(\mathcal{A}')=Aut(\mathcal{A}'')$ and every set is in bijection with a pure set.
    \newline $\mathcal{A}'\in\mathcal{N}$ so there is $\mathbf{a}\in\mathbb{A}^{<\omega}$ s.t. $\mathcal{M}\models$\textit{`$\mathbf{a}$ supports $\mathcal{A}'$ w.r.t. $Aut(\mathcal{A})$'}. The language of $\mathcal{A}'$ is pure so for all predicates $P$ in $\mathcal{A}'$, $\mathcal{A}'\in\mathcal{N}$ so there is $\mathbf{a}\in\mathbb{A}^{<\omega}$ s.t. $\mathcal{M}\models$\textit{`$\mathbf{a}$ supports $P$ w.r.t. $Aut(\mathcal{A})$'}. As \textit{$\mathcal{M}\models$`$\mathcal{A}$ is $\aleph_0$-homogeneous'} and by Corollary \ref{cor:supported_union_of_types}, \textit{$\mathcal{M}\models$`$P$ is a union of $\mathcal{A}$-types in parameter $\mathbf{a}$'}.
    \newline Let $T=\left\{tp_{\mathcal{A}}(\mathbf{a}^\frown\mathbf{b}):\mathbf{b}\in P\right\}$. As the language of $\mathcal{A}'$ is pure, $T$ is pure. As model truth is $\Delta^{ZFA}_1$ and $\mathcal{M}\models P=\left\{\mathbf{b}:\mathbf{a}^\frown\mathbf{b}\in \bigcup_{t\in T}t(\mathcal{A})\right\}$, $\mathcal{M}'\models P=\left\{\mathbf{b}:\mathbf{a}^\frown\mathbf{b}\in \bigcup_{t\in T}t(\mathcal{A})\right\}$. So $\mathcal{M}'\models$\textit{`$\mathbf{a}$ supports $P$ w.r.t. $Aut(\mathcal{A})$'}.
    \newline All predicates of $\mathcal{A}'$ are supported by $\mathbf{a}$ w.r.t. $\mathcal{A}$ in $\mathcal{M}'$, so, as $\mathcal{A}'$ has pure language, $\mathcal{A}'$ is supported by $\mathbf{a}$ w.r.t. $\mathcal{A}$ in $\mathcal{M}'$.
    Let:
    \begin{equation*}
        \mathcal{N}'=PM(\mathcal{M}',Aut(\mathcal{A}),\mathcal{F}^{Aut(\mathcal{A})}_{fin})^{\mathcal{M}'}
    \end{equation*}
    By the above, $\mathcal{A}'\in\mathcal{N}'$.
    Now, $\mathcal{A}\in\mathcal{N}$ and $\mathcal{A}'\in\mathcal{N}'$ so, by Corollary \ref{thm:structure_perm_model_eq}, $\mathcal{N}=\mathcal{N}'$.
\end{proof}

We now put our theorems together to conclude
\begin{theorem}
\label{thm:full_transfer}
    If:
    \begin{itemize}
        \item $\mathcal{M}\models ZFA+AC^{pure}+SVC(\mathbb{A}^{<\omega})$;
        \item $\mathcal{A}\in\mathcal{M}$;
        \begin{itemize}
            \item $\mathcal{M}\models\mathcal{A}$ is $\aleph_0$-homogeneous;
            \item $\mathcal{A}$ has pure language;
        \end{itemize}
        \item $\mathcal{N}:=PM(\mathcal{M},Aut(\mathcal{A}),\mathcal{F}^{Aut(\mathcal{A})}_{fin})^{\mathcal{M}}$; and
        \item Lemma \ref{thm:transfer_down} holds with $\kappa$ on $\mathcal{N}$;
    \end{itemize}
    then
    \begin{itemize}
        \item $\mathcal{N}$ has a forcing extension $\mathcal{N}[\Gamma]\models\left|\mathbb{A}\right|=\kappa$ with
        \begin{equation*}
            \mathcal{N}=PM(\mathcal{N}[\Gamma],Aut(\mathcal{A}),\mathcal{F}^{Aut(\mathcal{A})}_{fin})^{\mathcal{N}[\Gamma]}
        \end{equation*}
        \item there is a structure $\mathcal{A}^*\in\mathcal{V}$ s.t. $\left|\mathcal{A}^*\right|=\kappa$, $\mathcal{A^*}\equiv \mathcal{A}$ and
        \begin{equation*}
            \left\langle\mathcal{N},\mathcal{A}\right\rangle\equiv \left\langle PM( \mathcal{A}^*),\underline{\mathcal{A^*}}\right\rangle
        \end{equation*}
    \end{itemize}
\end{theorem}
\begin{remark}
    If $\mathcal{A}^*\cong\mathcal{B}^*$ then $\left\langle PM\left(\mathcal{A}^*\right),\underline{\mathcal{A}^*}\right\rangle\cong \left\langle PM\left(\mathcal{B}^*\right),\underline{\mathcal{B}^*}\right\rangle$ so if $Th(\mathcal{A})$ is $\kappa$-categorical then we can determine $\mathcal{A}^*$ up-to isomorphism.
\end{remark}
\begin{remark}
    If $\mathcal{M}\models\left\langle\mathcal{N},\mathcal{A}\right\rangle\cong \left\langle PM( \mathcal{B}^*),\underline{\mathcal{B^*}}\right\rangle$ then
    \begin{equation*}
        \left\langle PM(\mathcal{A}^*),\underline{\mathcal{A}^*}\right\rangle\equiv \left\langle PM(\mathcal{B}^*),\underline{\mathcal{B}^*}\right\rangle
    \end{equation*}
\end{remark}
\subsection{Examples}
\label{sec:examples}
Fix a pure universe $\mathcal{V}\models ZFC$.
\begin{example}[Non-Example]
    \label{examples:non-example}
    Let
    \begin{equation*}
        \mathcal{U}:=\left\langle \left(\mathbb{R}\oplus\mathbb{Q}\right)\otimes\mathbb{R},\leq\right\rangle
    \end{equation*}
    where $\oplus,\otimes$ are co-products and products in the category of total orders.
    \newline $PM(\mathcal{U})$ is not elementarily equivalent to a permutation model of $\mathcal{V}^*(\aleph_0)$. Indeed,
    \begin{equation}\label{eq:example_uncountable_statement}
        PM(\mathcal{U})\models\begin{array}{l}
             \textit{`there exists }\mathcal{S}\textit{, a set of disjoint,}\\
             \textit{ non-empty subsets of }\mathbb{A}\textit{ with }\left|\mathcal{S}\right|=2^{\aleph_0}\textit{'}
        \end{array}
    \end{equation}
    Namely,
    \begin{equation*}
        \left\{\underline{\mathbb{R}\times\left\{x\right\}}:x\in\mathbb{R}\right\}
    \end{equation*}
    Clearly this cannot occur in a permutation model of $\mathcal{V}^*(\aleph_0)$, as all partitions of the atoms are countable and the statement in \eqref{eq:example_uncountable_statement} is upwards absolute between a permutation model and the original model.
\end{example}

\begin{example}[$\left\langle PM(\left\langle\mathbb{R},\leq\right\rangle),\underline{\left\langle\mathbb{R},\leq\right\rangle}\right\rangle$ and $\left\langle PM(\left\langle\mathbb{Q},\leq\right\rangle),\underline{\left\langle\mathbb{Q},\leq\right\rangle}\right\rangle$]
\label{examples:R}
Let:
\begin{align*}
    \mathcal{Q}&:=\left\langle\mathbb{Q},\leq\right\rangle\\
    \mathcal{R}&:=\left\langle\mathbb{R},\leq\right\rangle
\end{align*}
As $Th(\mathcal{R})=Th(\mathcal{Q})$ is $\aleph_0$-categorical, if Theorem \ref{thm:full_transfer} applies (see Appendix \ref{sec:conditions_R}), then
\begin{equation*}
    \left\langle PM(\mathcal{R}),\underline{\mathcal{R}}\right\rangle\equiv\left\langle PM(\mathcal{Q}),\underline{\mathcal{Q}}\right\rangle
\end{equation*}
\end{example}

\begin{example}[$\left\langle PM(\left\langle\mathbb{R}\oplus\left\{\star\right\}\oplus\mathbb{Q},\leq\right\rangle),\underline{\left\langle\mathbb{R}\oplus\left\{\star\right\}\oplus\mathbb{Q},\leq\right\rangle}\right\rangle$ and $\left\langle PM(\left\langle\mathbb{Q},\leq\right\rangle),\underline{\left\langle\mathbb{Q},\leq\right\rangle}\right\rangle$]
\label{examples:R''}
Let:
\begin{align*}
    \mathcal{R}'&:=\left\langle\mathbb{R}\oplus\left\{\star\right\}\oplus\mathbb{Q},\leq\right\rangle\\
    \mathcal{R}''&:=\left\langle\mathbb{R}\oplus\left\{\star\right\}\oplus\mathbb{Q},\leq,\star\right\rangle\\
    \mathcal{Q}''&:=\left\langle\mathbb{Q},\leq,0\right\rangle
\end{align*}

As $Th(\mathcal{R}'')=Th(\mathcal{Q}'')$ is $\aleph_0$-categorical, if Theorem \ref{thm:full_transfer} applies (see Appendix \ref{sec:conditions_R''}), then
\begin{equation*}
    \left\langle PM(\mathcal{R}''),\underline{\mathcal{R}''}\right\rangle\equiv\left\langle PM(\mathcal{Q}''),\underline{\mathcal{Q}''}\right\rangle
\end{equation*}
Note that
\begin{align*}
    PM(\mathcal{R}')&=PM(\mathcal{R}'')\\
    PM(\mathcal{Q})&=PM(\mathcal{Q}'')
\end{align*}
And that $\mathcal{R}'$ can be defined from $\mathcal{R}''$ in $PM(\mathcal{R}'')$ by some formula (forgetting the constant), and the same for $\mathcal{Q}$ from $\mathcal{Q}''$ in $PM(\mathcal{Q}'')$ by the same formula. So
\begin{equation*}
    \left\langle PM(\mathcal{R}'),\underline{\mathcal{R}'}\right\rangle\equiv\left\langle PM(\mathcal{Q}),\underline{\mathcal{Q}}\right\rangle
\end{equation*}
\begin{remark}
    There is no class isomorphism in $\mathcal{V}^*(2^{\aleph_0})$
    \begin{equation*}
        PM(\mathcal{R})\cong PM(\mathcal{R}')
    \end{equation*}
    This is as $\underline{\mathbb{Q}}\in PM(\mathcal{R}')$ is countably infinite in $\mathcal{V}^*(2^{\aleph_0})$.
    \newline All subsets of atoms in $PM(\mathcal{R})$ are finite or uncountably infinite in $\mathcal{V}^*(2^{\aleph_0})$ and an isomorphism $PM(\mathcal{R})\cong PM(\mathcal{R}')$ would preserve the cardinality of sets, so $\underline{\mathbb{Q}}$ will have no image under an isomorphism. Therefore, there is no class isomorphism.
\end{remark}
\end{example}
\begin{example}
\label{examples:D}
Let:
\begin{align*}
    &\mathcal{Q'}:=\left\langle\mathbb{Q},\leq,x\mapsto x \leq\sqrt{2}\right\rangle\\
    &\mathcal{D}:=\left\langle\mathbb{R}\oplus\mathbb{Q},\leq\right\rangle\\
    &\mathcal{D}':=\left\langle\mathbb{R}\oplus\mathbb{Q},\leq,x\mapsto x\in\mathbb{R}\right\rangle
\end{align*}
As $Th(\mathcal{D}')=Th(\mathcal{Q}')$ is $\aleph_0$-categorical, if Theorem \ref{thm:full_transfer} applies (see Appendix \ref{sec:conditions_D}), then
\begin{equation*}
    \left\langle PM(\mathcal{D}'),\underline{\mathcal{D}'}\right\rangle\equiv\left\langle PM(\mathcal{Q}'),\underline{\mathcal{Q}'}\right\rangle
\end{equation*}
Note that
\begin{equation*}
    PM(\mathcal{D})=PM(\mathcal{D}')
\end{equation*}
So
\begin{equation*}
    PM(\mathcal{D})\equiv PM(\mathcal{Q}')
\end{equation*}
Although $\underline{\mathcal{D}}$ is definable from $\underline{\mathcal{D}'}$ and $\underline{\mathcal{Q}}$ is definable from $\underline{\mathcal{Q}'}$ by the same formula, we cannot conclude as in the Example \ref{examples:R''}, as $PM(\mathcal{Q})\not=PM(\mathcal{Q}')$. This is because they are not elementarily equivalent. In $PM(\mathcal{Q}')$, the atoms can be written as a disjoint union of models of $DLO$, which cannot be done in $PM(\mathcal{Q})$.
\end{example}
Figure \ref{fig:examples} visualises the examples from this section.
\begin{figure}[h!]
\centering
\caption{Examples \ref{examples:non-example}, \ref{examples:R}, \ref{examples:R''}, and \ref{examples:D}}
    \label{fig:examples}
    \includegraphics[page=1]{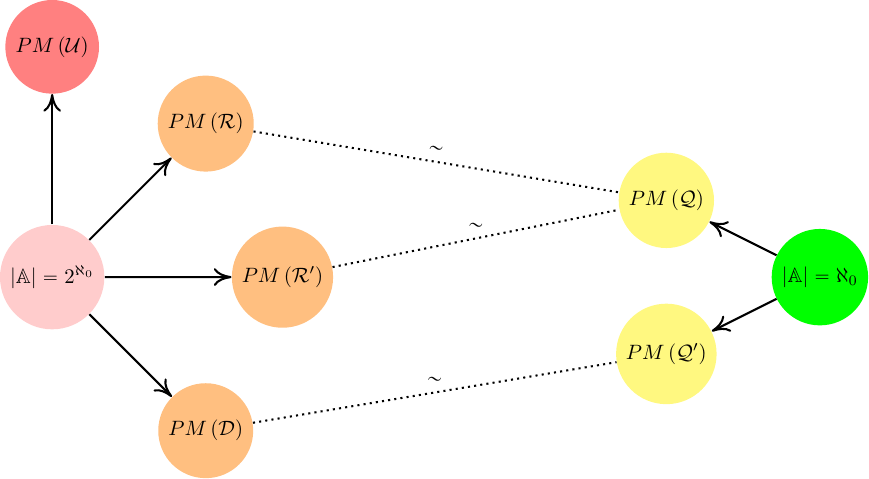}
\end{figure}

%% file: text/sec6-future.tex
\section{Conclusion and Future Work}

\subsection{Conclusion}
Our goal was to transfer theorems about orbit-finite constructions from well-understood structures to other structures.
\newline We have developed machinery to show that finite-support permutation models induced by structures of distinct cardinalities are elementarily equivalent (Theorem \ref{thm:full_transfer}).
\newline Numerous statements about orbit-finite constructions are expressible in finite-support permutation models (Remark \ref{rem:permutation_model_expressive}) and our understanding of orbit-finite constructions relies on countability of the inducing structure, so numerous statements about well-understood structures will hold for structures that induce elementarily equivalent permutation models e.g. orbit-finite constructions on $\left\langle \mathbb{Q},\leq\right\rangle$ are well-understood, allowing us to understand orbit-finite constructions on $\left\langle \mathbb{R},\leq\right\rangle$.
\newline There are obvious criticisms of this thesis. 
We rely on $AC$ e.g. Theorem \ref{thm:perm_forcing}; Theorem \ref{thm:finite_support_fo_improved}; Theorem \ref{thm:orbit_finite_internal}; and Theorem \ref{thm:inducing_structure_after_transfer}. However, our result doesn't seem like it should require choice.
\newline We lack precise analysis into the expressibility of permutation models.
\subsection{Future Work}
The above criticisms suggest the questions:
\begin{itemize}
    \item Do our results fail without $AC^{pure}$?
    \begin{itemize}
        \item Can we identify two first Fraenkel models on the same pure universe with different theories?
        \item Can we identify some explicit failure? Can $\left\langle PM(\mathcal{R}),\mathcal{R}\right\rangle\not \equiv\left\langle PM(\mathcal{Q}),\mathcal{Q}\right\rangle$ occur without assuming $AC^{pure}$?
        \item Is there some simpler proof without assuming $AC^{pure}$?
    \end{itemize}
    \item Can we express being a finite-support permutation model without assuming $SVC(\mathbb{A}^{<\omega})$?
    \item Can we express when a set is orbit-finite without using $AC$?
\end{itemize}
Theorem \ref{thm:first_fraenkel_unique} makes no assumption about choice principles, except $AC^{pure}$. However, it is currently unknown whether such a thing is possible.
\begin{itemize}
    \item Is there a model of $ZFA+AC^{pure}+\neg SVC$? \cite{hall_permutation_2007}
\end{itemize}
Work must be done to determine what statements about orbit-finite constructions are expressible in finite-support permutation models and a list of examples should be made.

%% file: text/appendix-examples-conditions.tex
\section{Checking Conditions for the Examples}
\subsection{$\mathcal{R}$}
\label{sec:conditions_R}

\begin{equation*}
    Th(\mathcal{R})=Th(\mathcal{Q})=DLO
\end{equation*}
The theory of unbounded, dense, linear orders is well-known to be $\aleph_0$-categorical.
\begin{claim}
    \begin{enumerate}
        \item \label{thm:R_hom}$\mathcal{R}$ is $\aleph_0$-homogeneous.
        \item \label{thm:R_restricts}For all $ m\in\mathbb{N}$, $\mathbf{a},\mathbf{b}\in\mathbb{Q}^m$ and $\sigma'\in Aut(\mathcal{R})$, if $\sigma'\mathbf{a}=\mathbf{b}$ then there exists $\sigma\in Aut(\mathcal{R})_\mathbb{Q}$ s.t. $\sigma\mathbf{a}=\mathbf{b}$.
        \item \label{thm:R_absorbs_Q}For all $\mathcal{S}\subseteq_{fin}\mathbb{Q}$ and $\mathcal{S}'\subseteq_{fin}\mathbb{R}$, there exists $\sigma\in Aut(\mathcal{R})$ s.t. if $q\in\mathcal{S}$ then $\sigma q=q$ and $\sigma[\mathcal{S}']\subseteq \mathbb{Q}$.
    \end{enumerate}
\end{claim}

\begin{proof}[\ref{thm:R_hom}]
    Suppose that $\mathbf{a},\mathbf{b}\in\mathbb{R}^n$ have the same $\mathcal{R}$-type.
    Let $\left(i_j\right)_{0\leq j\leq n-1}$ be a renumbering of $\left\{0,\dots,n-1\right\}$ s.t. $\left(a_{i_j}\right)_{0\leq j\leq n-1}$ and $\left(b_{i_j}\right)_{0\leq j\leq n-1}$ are increasing.
    Define $f\in Aut(\mathcal{R})$ by:
    \begin{equation*}
        f(x):=\left\{
        \begin{array}{ccc}
             b_{i_j}\frac{x-a_{i_{j+1}}}{a_{i_j}-a_{i_{j+1}}}+b_{i_{j+1}}\frac{x-a_{i_{j}}}{a_{i_{j+1}}-a_{i_{j}}}&a_{i_j}<x<a_{i_{j+1}}&0\leq j< n-1  \\
             b_{i_j}&x=a_{i_j}&0\leq j\leq n-1\\
             x+b_{i_0}-a_{i_0}&x<a_{i_0}&\\
             x+b_{i_{n-1}}-a_{i_{n-1}}&x>a_{i_{n-1}}&
        \end{array}
        \right.
    \end{equation*}
\end{proof}
\begin{proof}[\ref{thm:R_restricts}]
    Let $\mathbf{a},\mathbf{b}\in\mathbb{Q}^m$ and $\sigma\in Aut(\mathcal{R})$ with $\sigma\mathbf{a}=\mathbf{b}$. So $\mathbf{a}$ and $\mathbf{b}$ have the same $\mathcal{R}$-type.
    \newline Take $f\in Aut(\mathcal{R})$ s.t. $f\mathbf{a}=\mathbf{b}$ as in the proof of \eqref{thm:R_hom}; $f\in Aut(\mathcal{R})_\mathbb{Q}$ as the coefficients and endpoints of linear segments are rational.
\end{proof}

\begin{proof}[\ref{thm:R_absorbs_Q}]
    Let $\mathbf{a}\in\mathbb{Q}^n$ be strictly increasing and $r\in\mathbb{R}$. We consider the cases:
    \newline If $r\in \mathbf{a}$, let $\sigma=id_\mathbb{R}$.
    \newline If $r< a_0$, $\mathbf{R:}=\left(a_0-1\right)^\frown\mathbf{a}$ and $\mathbf{L}:=\left(r\right)^\frown \mathbf{a}$ are both strictly increasing.
    \newline If $r> a_{n-1}$, $\mathbf{R}:=\mathbf{a}^\frown\left(a_{n-1}+1\right)$ and $\mathbf{L}:=\mathbf{a}^\frown\left(r\right)$ are both strictly increasing.
    \newline If $a_i<r<a_{i+1}$, $\mathbf{R}:={\mathbf{a}_{\leq i}}^\frown\left(\frac{a_i+a_{i+1}}{2}\right)^\frown\mathbf{a}_{\geq i+1}$ and $\mathbf{L}:={\mathbf{a}_{\leq i}}^\frown\left(r\right)^\frown\mathbf{a}_{\geq i+1}$ are both strictly increasing.
    \newline $DLO$ has quantifier elimination and $\mathbf{L},\mathbf{R}$ satisfy the same atomic formulae, so they have the same type. By \eqref{thm:R_hom}, there is $\sigma\in Aut(\mathcal{R})$ s.t. $\sigma \mathbf{L}=\mathbf{R}$.
    This shows \eqref{thm:R_absorbs_Q} for $\left|\mathcal{S}'\right|=1$ and we induct to get \eqref{thm:R_absorbs_Q} for $\mathcal{S}'\subseteq_{fin}\mathbb{R}$.
\end{proof}
So $\mathcal{R}$ satisfies the conditions of Theorem \ref{thm:transfer_down_sub}

\subsection{$\mathcal{R''}$}
\label{sec:conditions_R''}
$DLO$ is $\aleph_0$-categorical in the extended language with a constant symbol $c$. This is a standard result from model theory.
\newline Now $\mathcal{R}'',\mathcal{Q}''\models DLO$, which is complete as $\aleph_0$-categorical, so $Th(\mathcal{R}'')=Th(\mathcal{Q}'')$.
\begin{claim}
    \begin{enumerate}
        \item \label{thm:R''_hom}$\mathcal{R}''$ is $\aleph_0$-homogeneous.
        \item \label{thm:R''_restricts}For all $m\in\mathbb{N}$, $\mathbf{a},\mathbf{b}\in\left(\mathbb{Q}\oplus\left\{\star\right\}\oplus\mathbb{Q}\right)^m$, and $\sigma'\in Aut(\mathcal{R}'')$, if $\sigma'\mathbf{a}=\mathbf{b}$ then there exists $\sigma\in Aut(\mathcal{R}'')_{\left(\mathbb{Q}\oplus\left\{\star\right\}\oplus\mathbb{Q}\right)}$ s.t. $\sigma\mathbf{a}=\mathbf{b}$.
        \item \label{thm:R''_absorbs_Q}For all $\mathcal{S}\subseteq_{fin}\left(\mathbb{Q}\oplus\left\{\star\right\}\oplus\mathbb{Q}\right)$ and $\mathcal{S}'\subseteq_{fin}\mathbb{R}$, there exists $\sigma\in Aut(\mathcal{R}'')$ s.t. for $q\in\mathcal{S}$, $\sigma q=q$ and $\sigma[\mathcal{S}']\subseteq \left(\mathbb{Q}\oplus\left\{\star\right\}\oplus\mathbb{Q}\right)$.
    \end{enumerate}
\end{claim}
\begin{proof}[\ref{thm:R''_hom}]
    Suppose that $\mathbf{a},\mathbf{b}\in\left(\mathbb{R}\oplus\left\{\star\right\}\oplus\mathbb{Q}\right)^n$.
    Let:
    \begin{align*}
        \mathbf{a}^l&=\left(a_i\right)_{a_i<\star}\\
        \mathbf{b}^l&=\left(b_i\right)_{b_i<\star}\\
        \mathbf{a}^r&=\left(a_i\right)_{a_i>\star}\\
        \mathbf{b}^r&=\left(b_i\right)_{b_i>\star}
    \end{align*}
    As the types are the same, the domains of $\mathbf{a}^l,\mathbf{b}^l$ match and the domains of $\mathbf{a}^r,\mathbf{b}^r$ match.
    \newline Now consider $\mathbf{a}^l,\mathbf{b}^l$ as tuples in $\mathcal{R}$, their types match (as formulae on $\mathcal{R}$ rise to formulae in $\mathcal{R}''$ by localising quantifiers to $<\star$) so there is automorphism $f\in Aut(\mathcal{R})$ with $f\mathbf{a}^l=\mathbf{b}^l$.
    \newline Similarly, there is automorphism $f'\in Aut(\mathcal{Q})$ with $f\mathbf{a}^r=\mathbf{b}^r$.
    \begin{equation*}
        f\cup f'\cup\left\{\left\langle\star,\star\right\rangle\right\}\in Aut(\mathcal{R}'')
    \end{equation*}
    This maps $\mathbf{a}$ to $\mathbf{b}$.
\end{proof}
\begin{proof}[\ref{thm:R''_restricts}]
    Let $\mathbf{a},\mathbf{b}\in\left(\mathbb{Q}\oplus\left\{\star\right\}\oplus\mathbb{Q}\right)^m,\sigma'\in Aut(\mathcal{R}'')$ s.t. $\sigma' \mathbf{a}=\mathbf{b}$. Let
    \begin{align*}
        \mathbf{a}^l&=\left(a_i\right)_{a_i<\star}\\
        \mathbf{b}^l&=\left(b_i\right)_{b_i<\star}
    \end{align*}
    They have the same $\mathcal{R}$-type, so by \ref{sec:conditions_R}, find $f\in Aut(\mathcal{R})_\mathbb{Q}$ with $f\mathbf{a}^l=\mathbf{b}^l$.
    \begin{equation*}
        f\cup\left(\sigma'\restriction\left(\emptyset\oplus\left\{\star\right\}\oplus\mathbb{Q}\right)\right)\in Aut(\mathcal{R}'')_{\left(\mathbb{Q}\oplus\left\{\star\right\}\oplus\mathbb{Q}\right)}
    \end{equation*}
    This maps $\mathbf{a}$ to $\mathbf{b}$.
\end{proof}
\begin{proof}[\ref{thm:R''_absorbs_Q}]
    Let $\mathbf{a}\in\left(\mathbb{Q}\oplus\emptyset\oplus\emptyset\right)^n$ and $r\in\left(\mathbb{R}\setminus\mathbb{Q}\oplus\emptyset\oplus\emptyset\right)$.
    As in \ref{sec:conditions_R}, find $f\in Aut(\mathcal{R})$ s.t. $f\mathbf{a}=\mathbf{a}$ and $fr\in\left(\mathbb{Q}\oplus\emptyset\oplus\emptyset\right)$.
    \begin{equation*}
        f\cup id_{\left(\mathbb{\emptyset}\oplus\left\{\star\right\}\oplus\mathbb{Q}\right)}\in Aut(\mathcal{R}'')
    \end{equation*}
    This fixes $\mathbf{a}, \star,$ and $\emptyset\oplus\emptyset\oplus\mathbb{Q}$. Applying this repeatedly by induction gives us the result.
\end{proof}
\subsection{$\mathcal{D}$}
\label{sec:conditions_D}
The proofs here are almost the same as \ref{sec:conditions_R''}, but removing references to $\star$.
So it suffices to show that $Th(\mathcal{D}')=Th(\mathcal{Q}')$ is $\aleph_0$-categorical and that $PM(\mathcal{D})=PM(\mathcal{D}')$.
The later is obvious: predicates of $\mathcal{D}$ and predicates of $\mathcal{D}'$ are mutually equivariant, since their automorphism groups fix $\mathbb{R}\oplus\emptyset$ and $\emptyset\oplus\mathbb{Q}$; they have pure language; apply Theorem \ref{thm:different_structure_model}.
\newline For the latter, let $T$ be the theory in the language with binary predicate $\leq$ and unary predicate $P$ saying that:
\begin{itemize}
    \item If $P(x)$ and $\neg P(y)$ then $x\leq y$
    \item $\leq$ is a $DLO$ on $P$
    \item $\leq$ is a $DLO$ on $\neg P$
\end{itemize}
Note that $\mathcal{D}', \mathcal{Q}'\models T$.
Let $\mathcal{A},\mathcal{B}\models T$ be countable.
Let:
\begin{align*}
    \mathcal{A}^l&:=\left\langle P^\mathcal{A},\leq^\mathcal{A}\right\rangle\\
    \mathcal{A}^r&:=\left\langle A\setminus P^\mathcal{A},\leq^\mathcal{A}\right\rangle\\
    \mathcal{B}^l&:=\left\langle P^\mathcal{B},\leq^\mathcal{B}\right\rangle\\
    \mathcal{B}^r&:=\left\langle B\setminus P^\mathcal{B},\leq^\mathcal{B}\right\rangle\\
    \mathcal{A}^l,\mathcal{A}^r,\mathcal{B}^l,\mathcal{B}^r&\models DLO
\end{align*}
So there are isomorphisms
\begin{align*}
    f&:\mathcal{A}^l\cong\mathcal{B}^l\\
    g&:\mathcal{A}^r\cong\mathcal{B}^r
\end{align*}
So
\begin{equation*}
    f\cup g:\mathcal{A}\cong\mathcal{B}
\end{equation*}
So $T$ is $\aleph_0$-categorical and complete.